\documentclass[11pt, reqno]{amsart}
\usepackage{graphicx,psfrag}
\usepackage[psamsfonts]{amssymb}
\usepackage{pdfsync}
\usepackage{amscd}
\usepackage{amsmath}
\usepackage{amsxtra}
\usepackage{mathrsfs}
\usepackage{stmaryrd}
\usepackage[small,nohug,heads=littlevee]{diagrams}
\usepackage{comment}
\usepackage{enumerate}
\usepackage{xfrac}
\usepackage{hyperref}
\diagramstyle[labelstyle=\scriptstyle]

\usepackage{mathpple}

\usepackage[width=6.3in, height=8.5in, bottom=1.3in, centering]{geometry}

\theoremstyle{plain}
    \newtheorem{theorem}{Theorem}[section]
    \newtheorem{lemma}[theorem]{Lemma}
    \newtheorem{corollary}[theorem]{Corollary}
    \newtheorem{proposition}[theorem]{Proposition}

\theoremstyle{definition}
    \newtheorem{definition}[theorem]{Definition}

    \newtheorem*{thank}{Acknowledgements}
\theoremstyle{remark}
    \newtheorem{remark}[theorem]{Remark}

\numberwithin{equation}{section}

\newcommand{\ZZ}{\mathbb{Z}}

\newcommand{\NN}{\mathcal{N}}

\newcommand{\bNconn}{\overline{\mathbb{N}}}
\newcommand{\bNconnd}{\overline{\mathbb{N}}^\vee}
\newcommand{\NNconn}{\mathfrak{N}}

\newcommand{\Conn}{\mathfrak{Conn}}

\newcommand{\complex}{\mathbb{C}}
\newcommand{\nat}{\mathbb{N}}

\newcommand{\A}{\mathcal{A}}

\newcommand{\F}{\mathcal{F}}

\newcommand{\AOO}{\A^{0,0}}

\newcommand{\AOD}{\A^{0,\bullet}}

\newcommand{\PA}{\mathcal{P}_{A}}

\newcommand{\gext}{\operatorname{Ext}}

\newcommand{\Hom}{\operatorname{Hom}}
\newcommand{\End}{\operatorname{End}}
\newcommand{\Id}{\operatorname{Id}}
\newcommand{\dbar}{\overline{\partial}}
\newcommand{\expmap}{\operatorname{exp}}
\newcommand{\normal}{N}
\newcommand{\conormal}{N^\vee}

\newcommand{\btnormal}{\overline{N}_\otimes}

\newcommand{\bcnormal}{\overline{\mathcal{N}}}

\newcommand{\dbarsharp}{\dbar^\sharp}

\newcommand{\substr}{\operatorname{Substr}}
\newcommand{\sgn}{\operatorname{sgn}}

\newcommand{\Cinf}{$C^\infty$}

\newcommand{\aaa}{\mathfrak{a}}

\newcommand{\Diagfinite}{\Delta^{\upscript{(r)}}}
\newcommand{\Xdiag}{X_{X \times X}\formal}
\newcommand{\Xdiagfinite}{X^{(r)}_{X \times X}}
\newcommand{\Shat}{\hat{S}}
\newcommand{\projtan}{\tau}
\newcommand{\liftnormal}{\rho}
\newcommand{\normalconn}{\nabla^\bot}
\newcommand{\normalconnbar}{\overline{\nabla}^\bot}

\newcommand{\Dnormal}{\mathfrak{D}}
\newcommand{\symmconn}{\overline{\nabla}}
\newcommand{\Sym}{\operatorname{Sym}}
\newcommand{\Sh}{\operatorname{Sh}}
\newcommand{\Aut}{\operatorname{Aut}}
\newcommand{\Spf}{\operatorname{Spf}}
\newcommand{\GL}{\mathbf{GL}}

\newcommand{\JetX}{\mathcal{J}^\infty_X}
\newcommand{\JetXfinite}{\mathcal{J}^r_X}
\newcommand{\graded}{\operatorname{gr}}
\newcommand{\Lie}{\mathcal{L}}

\newcommand{\Ker}{\operatorname{Ker}}
\newcommand{\upscript}[1]{{\scriptscriptstyle{#1}}}

\newcommand{\Linf}{$L_\infty$}
\newcommand{\Linfs}{$L_\infty[1]$}

\newcommand{\Yhat}{{\hat{Y}}}
\newcommand{\Yhatfinite}{{\hat{Y}^{\upscript{(r)}}}}

\newcommand{\Osheaf}{\mathscr{O}}

\newcommand{\Asheaf}{\mathscr{A}}
\newcommand{\Isheaf}{\mathscr{I}}

\newcommand{\formal}{^{\upscript{(\infty)}}}

\newcommand{\dash}{\operatorname{-}}
\newcommand{\proj}{\operatorname{pr}}
\newcommand{\Der}{\mathcal{D}er}
\newcommand{\field}{\mathbb{K}}
\newcommand{\LRs}{$LR_\infty[1]$}
\newcommand{\LD}{L^\bullet}
\newcommand{\AD}{\A^\bullet}
\newcommand{\ctensor}{\hat{\otimes}}
\newcommand{\ShatL}{\Shat_{\A}(L,\A)}
\newcommand{\shape}{S_N}

\newarrow{Equal} =====

\title[Dolbeault dga and  $L_\infty$-algebroid of the formal neighborhood]
{Dolbeault dga and  $L_\infty$-algebroid of the formal neighborhood}

\author{Shilin Yu}
\address{Department of Mathematics, University of Pennsylvania, PA 19104-6395, USA
}
\email{shilinyu@math.upenn.edu}


\thanks{This research was partially supported by the grant DMS1101382  from the National Science Foundation.}


\begin{document}

\begin{abstract}
 We continue the study the Dolbeault dga of the formal neighborhood of an arbitary closed embedding of complex manifolds previously defined by the author in \cite{DolbeaultDGA}. The special case of the diagonal embedding has been studied in \cite{Diagonal}. We describe the Dolbeault dga explicitly in terms of the formal differential geometry of the embedding. Moreover, we show that the Dolbeault dga is the completed Chevalley-Eilenberg dga an $L_\infty$-algebroid structure on the shifted normal bundle of the submanifold. This generlizes the result of Kapranov on the diagonal embedding and Atiyah class.
\end{abstract}

\maketitle

\tableofcontents

\section{Introduction}

This is the continuation of \cite{DolbeaultDGA} and \cite{Diagonal}. In \cite{DolbeaultDGA} we introduced the notion of \emph{Dolbeaut differential graded algebra (dga)} of the formal neighborhood of a closed embedding of complex manifolds, which contains all the formal geometric information of the embedding. Then in \cite{Diagonal} we studied the Dolbeault dga of the diagonal embedding and recovered Kapranov's description of the formal neighborhood of the diagonal in terms of the Atiyah class (\cite{Kapranov}). In the current paper, we will generalize the results in \cite{Diagonal} to arbitrary closed embeddings and show how to describe the Dolbeault dga explicitly in terms of the differential geometry of the submanifold, at least when the ambient manifold has a K\"ahler metric.

The Dolbeault dga $A=(\A^\bullet(\Yhat),\dbar)$ for the formal neighborhood $\Yhat$ of any closed embedding $i: X \hookrightarrow Y$ is defined in a canonical way, independent of any auxiliary geometric structures of the manifolds. A certain dg category $\PA$ of certain dg-modules over $\A$ was built in \cite{DolbeaultDGA} following the work of \cite{Block1} and was shown to be a dg-enhancement of the derived category of coherent sheaves over $\Yhat$.

However, we are interested in explicit construction of objects in $\PA$, among which the most important one for us is the derived direct image of $\Osheaf_X$ on $\Yhat$, which will be the main content of our upcoming work \cite{Todd}. There we will study the quantized analytic cycle class defined by Grivaux \cite{GrivauxHKR}, which specializes to the usual Todd class in the case of diagonal embedding. For this purpose, we need a geometric description of the Dolbeault dga, which reflects how the submanifold 'curls' in the ambient manifold.

The paper \cite{Diagonal} provides such a description in the case of the diagonal embedding $\Delta: X \hookrightarrow X \times X$. It was shown that there exist isomorphisms between the Dolbeault dga of the formal neighborhood $\Xdiag$ and the dga $(\AOD_X(\Shat(T^*X)), D_\sigma)$ of the Dolbeault resolution of the completed symmetric algebra of the cotangent bundle of $X$, where the differentials $D_\sigma$ depend on sections $\sigma$ of certian jet bundle with infinite dimensional fibers and related to the Atiyah class of $X$. In the current paper, we generalize this result to the case of a general embedding $i: X \hookrightarrow Y$, i.e., we show that there are isomorphisms between the dgas $(\A^\bullet(\Yhat),\dbar)$ and $(\AOD_X(\Shat(\conormal)), \Dnormal)$, where $\conormal$ is the conormal bundle of the submanifold and the differential $\Dnormal$ again depends on sections of some jet bundle. The main idea is to consider the graph $\widetilde{i}: X \hookrightarrow X \times Y$ of $i$, which again is an embedding, and the natural map between the pairs $(X, X \times Y) \to (X, Y)$. The formal neighborhood $\Yhat$ of $X$ inside Y can then be studied by studying the formal neighborhood $X\formal_{X \times Y}$ of $X$ inside $X \times Y$. The latter has a similar description as that of the diagonal embedding (Theorem \ref{thm:formal_graph}).

Kapranov's original result was formulated as an \Linf-algebra structure on the shifted tangent bundle $TX[-1]$, whose binary bracket is given by the Atiyah class. In our language, the completed Chevalley-Eilenberg dga of $TX[-1]$ is the Dolbeault dga $\A^\bullet(\Xdiag)$ or $(\AOD_X(\Shat(T^*X)), D_\sigma)$. In particular, it induces an Lie algebra structure on $TX[-1]$ as an object in the derived category of $X$. 
%This can be used to explain why the Todd class is defined by the same function which appear in the differential of the exponential map of a Lie group (cite). 
Our result on general embeddings can also be reformulated as an \Linf-structure on the shifted normal bundle $N[-1]$. The novel discovery here is an extra $\infty$-anchor map $N[-1] \to TX$ which makes $N[-1]$ into an \emph{\Linf-algebroid}. In the case of the diagonal embedding, it recovers Kapranov's \Linf-algebra, where the anchor map vanishes. To our knowledge, the notion of \Linf-algebroid was first defined in the work of Kjeseth (\cite{Kjeseth1}, \cite{Kjeseth2}) under the name of strong homotopy Lie-Rinehart algebras. \Linf-algebroids also appear in other context, such as string theory (\cite{Stasheff}) and the study of foliations (\cite{Luca}). 

We want to mention that Calaque, C\u{a}ld\u{a}raru and Tu have established similar results in the algebraic setting (\cite{CalaqueCaldararuTu}). While they built a dg-Lie algebroid on some dg-sheaf which is quasi-isomorphic to $i_*N[-1]$ in the derived category of $Y$, our \Linf-algebroid has the Dolbeault resolution of the normal bundle as the underlying complex and the higher brackets do not vanish in general.

The paper is organized as follows. In \S~\ref{sec:Dolbeault_DGA} we recall the general definition of the Dolbeault dga $\A^\bullet(\Yhat)$ for the formal neighborhood $\Yhat$ of an arbitrary closed holomorphic embedding $i: X \hookrightarrow Y$ of complex manifolds. In \S~\ref{sec:Diagonal_Formal} we briefly review our reformulation (\cite{Diagonal}) of Kapranov's result of the diagonal embedding. We recall various infinite dimensional fiber bundles arising from formal geometry, which we already used heavily in \cite{Diagonal} to derive Kapranov's results. We then apply them in \ref{sec:General_Case} to get a description of the Dolbeault dga of an arbitrary embedding $i: X \hookrightarrow Y$, in which many other differential geometric quantities other than the curvature, such as Kodaira-Spencer class and shape operator, come into the picture. The main result is Theorem \ref{thm:MAIN_derivation}. For convenience we only discuss the K\"ahler case, yet the formulas make sense in broader context (see the comment at the beginning of \S~\ref{subsubsec:pi}) . We construct an isomorphism from our canonical yet abstractly defined Dolbeault dga $\A^\bullet(\Yhat)$ to a concrete dga, namely the completed symmetric algebra $\AOD(\Shat(\conormal))$ of the conormal bundle of $X$ in $Y$, and compute the differential on it. The main result is Theorem \ref{thm:MAIN_derivation} for the final answer. Finally, \S~\ref{sec:sh-LieAlgebroid} is contributed to the equivalent \Linf-description. We will recall the basic definitions of \Linf-algebroids from \cite{Luca}. For convenience, we will mainly use \Linfs-algebroid, which is a shifted version of \Linf-algebroid, since the signs involved in the formula are enormously simplified.

\begin{thank}
  The author would like to thank Jonathan Block, Damien Calaque, Andrei C\u{a}ld\u{a}raru, Nigel Higson and Junwu Tu for many discussions.This research was partially supported under NSF grant DMS-1101382.
\end{thank}

\section{Dolbeault dga of formal neighborhoods}\label{sec:Dolbeault_DGA}

\subsection{Definitions and notations}\label{subsec:defn}

We review the notations and definitions from \cite{DolbeaultDGA}. Let $i: (X, \Osheaf_X) \hookrightarrow (Y, \Osheaf_Y)$ be a closed embedding of complex manifolds where $\Osheaf_X$ and $\Osheaf_Y$ are the structure sheaf of germs of holomorphic functions over the complex manifolds $X$ and $Y$ respectively. Let $\Isheaf$ the ideal sheaf of $\Osheaf_Y$ of holomorphic functions vanishing along $X$. The main objects studied by this paper are \emph{the $r$-th formal neighborhood $\Yhatfinite$} of $X$ in $Y$, which is defined as the ringed space $(X, \Osheaf_{\Yhatfinite})$ whose the structure sheaf is
\begin{displaymath}
  \Osheaf_{\Yhatfinite} = \Osheaf_Y / \Isheaf^{r+1},
\end{displaymath}
and \emph{the (complete) formal neighborhood $\Yhat = \Yhat^{\upscript{(\infty)}}$}, which is defined to be the ringed space $(X, \Osheaf_{\Yhat})$ where
\begin{displaymath}
 \Osheaf_{\Yhat} = \varprojlim_{r} \Osheaf_{\Yhatfinite} = \varprojlim_{r} \Osheaf_X / \Isheaf^{r+1}.
\end{displaymath}
We will also use the notations $X^{\upscript{(\infty)}}_Y = \Yhat$ and $X^{\upscript{(r)}}_Y=\Yhatfinite$ when we need to emphasize the submanifolds.

In  \cite{DolbeaultDGA} \emph{the Dolbeault differential graded algebra (dga)} of the embedding $i: X \hookrightarrow Y$ is defined as follows. Let $(\AOD(Y), \dbar) = (\wedge^\bullet \Omega_Y^{0,1}, \dbar)$ be the Dolbeault complex of $Y$, thought of as a dga. For each nonnegative integer $r$, $\aaa^\bullet_r$ is set to be the graded ideal of $\AOD(Y)$ consisting of those forms $\omega \in \AOD(Y)$ satisfying
\begin{equation}\label{defn:aaa_finite}
  i^*(\Lie_{V_1} \Lie_{V_2} \cdots \Lie_{V_l} \omega) = 0, \quad \forall~ 1 \leq j \leq l,
\end{equation}
for any collection of smooth $(1,0)$-vector fields $V_1, V_2, \ldots, V_l$ over $Y$, where $0 \leq l \leq r$. By Proposition 2.1, \cite{DolbeaultDGA}, $\aaa^\bullet_r$ is invariant under the action of $\dbar$ and hence is a dg-ideal of $(\AOD(Y),\dbar)$.

\begin{definition}[Definition 2.3, \cite{DolbeaultDGA}]
  The \emph{Dolbeault dga of the $r$-th formal neighborhood $\Yhatfinite$} is the quotient dga
  \begin{displaymath}
    \A^\bullet(\Yhatfinite) := \AOD(Y) / \aaa^\bullet_r.
  \end{displaymath}
  The \emph{Dolbeault dga of the complete formal neighorhood $\Yhat$} is defined to be the inverse limit
  \begin{displaymath}
    \A^\bullet(\Yhat) = \A^\bullet(\Yhat^{\upscript{(\infty)}}):= \varprojlim_r \A^\bullet(\Yhatfinite).
  \end{displaymath}
  We will write $\A(\Yhat) = \A^0(\Yhat)$ and $\A(\Yhatfinite) = \A^0(\Yhatfinite)$ for the zeroth components of the Dolbeault dgas.
\end{definition}

\begin{comment}
By Remark 2.4., \cite{DolbeaultDGA}, we have the following alternative description of $(\A^\bullet(\Yhat),\dbar)$.

\begin{proposition}
  The natural map
  \begin{equation}\label{eq:Ahat_quotient}
    \AOD(Y) \Big/ \bigcap_{r \in \NN} \aaa^\bullet_r \to \A^\bullet(\Yhat)
  \end{equation}
  induced by the quotient maps
  \begin{displaymath}
    \AOD(Y) \to \A^\bullet(\Yhatfinite), \quad r \in \NN,
  \end{displaymath}
  is an isomorphism of dgas.
\end{proposition}
\end{comment}

The Dolbeault dga $\A^\bullet(\Yhatfinite)$ can be sheafified to a soft sheaf of dgas $\Asheaf^\bullet(\Yhatfinite)$ over $X$ for $r \in \nat$ or $r=\infty$ (see \cite{DolbeaultDGA} for details). Moreover, there are natural inclusions of sheaves of algebras $\Osheaf_{\Yhatfinite} \hookrightarrow \Asheaf(\Yhatfinite)$ The following result was proved in \cite{DolbeaultDGA}.

\begin{theorem}[Prop. 2.8., \cite{DolbeaultDGA}]
  For any nonnegative integer $r$ or $r = \infty$, the complex of sheaves
  \begin{displaymath}
    0 \to \Osheaf_{\Yhatfinite} \to \Asheaf^0_{\Yhatfinite} \xrightarrow{\dbar} \Asheaf^1_{\Yhatfinite} \xrightarrow{\dbar} \cdots \xrightarrow{\dbar} \Asheaf^m_{\Yhatfinite} \to 0
  \end{displaymath}
  is exact, where $m = \dim X$. In other words, $(\Asheaf^\bullet_{\Yhatfinite},\dbar)$ is a soft resolution of $\Osheaf_{\Yhatfinite}$.
\end{theorem}

\begin{comment}
The dga $\A^\bullet(\Yhat)$ is filtrated by the descending sequence of dg-ideals
\begin{displaymath}
  \A^\bullet(\Yhat) \supset \bbb^\bullet_0 \supset \bbb^\bullet_1 \supset \cdots,
\end{displaymath}
where $\bbb^\bullet_r$ is the kernel of the natural projection $\pi_r: \A^\bullet(\Yhat) \to \A^\bullet(\Yhatfinite)$. We can also sheafise $\bbb^\bullet_r$ to obtain sheaves of dg-ideals $\widetilde{\bbb}^\bullet_r$, as we did for $\A^\bullet(\Yhat)$. We have natural isomorphisms
\begin{equation}\label{eq:isom_aaa_bbb}
  \widetilde{\bbb}^\bullet_r / \widetilde{\bbb}^\bullet_{r+1} \simeq \widetilde{\aaa}^\bullet_r / \widetilde{\aaa}^\bullet_{r+1}.
\end{equation}
Note that both sides of \eqref{eq:isom_aaa_bbb} admit natural actions of $\Asheaf^{0,\bullet}_X \simeq \Asheaf^\bullet_{\Yhat} / \widetilde{\bbb}^\bullet_0 \simeq \Asheaf^{0,\bullet}_Y / \widetilde{\aaa}^\bullet_0$ and \eqref{eq:isom_aaa_bbb} is in fact an isomorphism of dg-modules over $(\Asheaf^{0,\bullet}_X, \dbar)$. Thus the associated graded algebra of $\A^\bullet(\Yhat)$ is isomorphic to the associated graded algebra of $\A^\bullet()$
\end{comment}

As the completion of $\AOD(Y)$ with respect to the filtration $\aaa^\bullet_r$, the dga $\A^\bullet(\Yhat)$ is itself filtered and its associated graded dga is
\begin{displaymath}
  \graded \A^\bullet(\Yhat) \simeq (\AOD(Y) / \aaa^\bullet_0) \oplus \bigoplus_{r=0}^{\infty} \aaa^\bullet_r / \aaa^\bullet_{r+1}.
\end{displaymath}
Note that $\AOD(X) \simeq \AOD(Y) / \aaa^\bullet_0$ and $\aaa^\bullet_r / \aaa^\bullet_{r+1}$ are dg-modules over $(\AOD(X),\dbar)$. We define, for each $r \geq 0$, a `cosymbol map' of complexes
\begin{equation}\label{eq:isom_assoc_graded}
  \tau_r : (\aaa^\bullet_r / \aaa^\bullet_{r+1}) \to \AOD_X(S^{r+1} \conormal),
\end{equation}
where $S^{r+1} \conormal$ is the $(r+1)$-fold symmetric tensor of the conormal bundle $\conormal$ of the embedding. Given any $(r+1)$-tuple of smooth sections $\mu_1, \ldots, \mu_{r+1}$ of $\normal$, we lift them to smooth sections of $TY|_X$ and extend to smooth $(1,0)$-tangent vector fields $\tilde{\mu}_1, \ldots, \tilde{\mu}_{r+1}$ on $Y$ (defined near $X$). We then define the image of $\omega + \aaa^\bullet_{r+1} \in \aaa^\bullet_r / \aaa^\bullet_{r+1}$ under $\tau$ for any $\omega \in \aaa^\bullet_{r}$, thought of as linear functionals on $(\normal)^{\otimes (r+1)}$, by the formula
\begin{equation}
  \tau_r (\omega + \aaa^\bullet_{r+1}) (\mu_1 \otimes \cdots \otimes \mu_{r+1}) = i^* \Lie_{\tilde{\mu}_1} \Lie_{\tilde{\mu}_2} \cdots \Lie_{\tilde{\mu}_{r+1}} \omega.
\end{equation}
The map $\tau_r$ is well-defined and is independent of the choice of the representative $\omega$ and $\tilde{\mu}_j$'s. Moreover, the tensor part of $\tau_r (\omega + \aaa^\bullet_{r+1})$ is indeed symmetric.

\begin{proposition}
  The map $\tau_r$ in \eqref{eq:isom_assoc_graded} is an isomorphism of dg-modules over $(\AOD(X),\dbar)$.
\end{proposition}

\begin{corollary}
  We have a natural isomorphism of dgas
  \begin{displaymath}
    \graded \A^\bullet(\Yhat) \simeq \bigoplus_{n=0}^{\infty} \AOD_X(S^n \conormal).
  \end{displaymath}
\end{corollary}

\begin{comment}
\begin{remark}
 Since $\Osheaf_{\Yhat}$ is defined as an inverse limit sheaf, it has a natural descending filtration and so $\Yhat$ should be regarded not only as a ringed space but also as a topologically ringed space. Similarly, defined as an inverse limit, the Dolbeault dga $\A^\bullet(\Yhat)$ is a filtered dga. Moreover, since $\AOD(Y)$ is a Fr\'echet dga and all ideals $\aaa^\bullet_r$ are closed, the inverse limit $\A^\bullet(\Yhat)$ can be made into a Fr\'echet dga equipped with the initial topology (which concides with the quotient Fr\'echet topology via the quotient map \eqref{eq:Ahat_quotient}). One can recover the formal neighborhood $\Yhat$ from the topological dga $\A^\bullet(\Yhat)$. The topology or the filtration would matter when one wants to consider morphisms between two such dgas or dg-modules over the Dolbeault dgas (see \cite{DolbeaultDGA}). However, these are not among the topics of the current paper. The reader only needs to keep in mind that  in this paper every morphism between two filtered dgas will preserve the filtrations. 
\end{remark}
\end{comment}

\section{Dolbeault dgas of diagonal embeddings}\label{sec:Diagonal_Formal}

Among all important and interesting examples is the diagonal embedding $\Delta: X \hookrightarrow X \times X$ of a complex manifold $X$ into product of two copies of itself. We then have the formal neighborhood of the diagonal, $\Xdiag$, which is understood as the dga $(\AOD(\Xdiag),\dbar)$ constructed in the previous section. It is isomorphic to the Dolbeault resolution of the infinite holomorphic jet bundle $\JetX$. In this case one can write down the $\dbar$-derivation explicitly (at least when $X$ is K\"ahler) due to a theorem below by Kapranov (\cite{Kapranov}), of which we will reproduce the K\"ahler case in a slightly different way and discuss the general situation later.

Intuitively one would expect that there is an isomorphism
\begin{equation}\label{eq:formal_isom}
  \AOD(\Xdiag) \simeq \AOD(\Shat^\bullet_X(T^*X))
\end{equation}
by taking 'Taylor expansions', where $\Shat^\bullet(T^*X)$ is the bundle of complete symmetric tensor algebra generated by the cotangent bundle of $X$ (which is natural identified with the conormal bundle of the diagonal embedding). Such an isomorphism does exist, but there is no canonical way to define it since one need to first choose some local coordinates to get Taylor expansions. Indeed we will see that the isomorphism depends (in a more or less 1-1 manner) on a smooth choice of formal (holomorphic) coordinates on $X$.

\subsection{Diagonal embeddings and jet bundles}\label{subsec:diag_jet}

We consider the case of diagonal embeddings. Let $X$ be a complex manifold and let $\Delta: X \hookrightarrow X \times X$ be the diagonal map. For convenience, we write $\Diagfinite = \Xdiagfinite$ for $r \in \nat$ or $r=\infty$ throughout this section.Denote by $\proj_1, \proj_2 : X \times X \to X$ the projections onto the first and second component of $X \times X$ respectively. The jet bundle $\JetXfinite$ of order $r$ ($r \in \nat$ or $r=\infty$) can be viewed as the sheaf of algebras
\begin{displaymath}
  \JetXfinite = \proj_{1*} \Osheaf_{\Diagfinite},
\end{displaymath}
which is a sheaf of $\Osheaf_X$-modules where the $\Osheaf_X$-actions are induced from the projection $\proj_1$. The sheaf $(\Asheaf^{0,\bullet}(\JetXfinite), \dbar)$ of Dolbeault complexes of $\JetXfinite$ is a sheaf of dgas and its global sections forms a dga
\begin{displaymath}
  \AOD(\JetXfinite) = \Gamma(X, \Asheaf^{0,\bullet}(\JetXfinite)).
\end{displaymath}
Since $\JetXfinite$ is a sheaf of $\Osheaf_X$-modules, $\AOD(\JetXfinite)$ is a dga over the dga $(\AOD(X),\dbar)$.

The Dolbeault dga $(\A^\bullet(\Diagfinite),\dbar)$ of the formal neighborhood $(\A^\bullet(\Diagfinite),\dbar)$ ($r \in \nat$ or $r=\infty$) is also an $(\AOD(X),\dbar)$-dga via the compositions of homomorphisms of dgas
\begin{displaymath}
  \AOD(X) \xrightarrow{\proj_1^*} \AOD(X \times X) \to \A^\bullet(\Diagfinite).
\end{displaymath}

\begin{proposition}[Prop. 2.8., \cite{Diagonal}]\label{prop:isom_diag_dga}
  The natural inclusion $\Osheaf_{\Diagfinite} \hookrightarrow \Asheaf_{\Diagfinite}$ determines an $(\AOD(X),\dbar)$-linear morphism of dgas
  \begin{displaymath}
    I_r: \AOD(\JetXfinite) \xrightarrow{\simeq} \A^\bullet(\Diagfinite)
  \end{displaymath}
  either when $r \in \nat$ or $r=\infty$. Similar results hold for the corresponding sheaves.
\end{proposition}

\subsection{Formal geometry}

\subsubsection{Differential geometry on formal discs}
\label{subsec:Formal_Discs}

Fix a complex vector space $V$ of dimension $n$. The formal disc $\widehat{V}$ is the formal neighborhood of $0$ in $V$. Its function algebra is the formal power series algebra
\begin{displaymath}
  \F = \complex \llbracket V^* \rrbracket = \Shat(V^*) = \prod_{i \geq 0} S^i V^*.
\end{displaymath}
It is a complete regular local algebra with a unique maximal ideal $\mathfrak{m}$ consisting of formal power series with vanishing constant term. The associated graded algebra with respect to the usual $\mathfrak{m}$-filtration is the symmetric algebra
\begin{displaymath}
  \graded \F = S(V^*) = \bigoplus_{i \geq 0} S^i V^*.
\end{displaymath}
Since we are in the complex analytic situation, we endow $\F$ with the canonical Fr\'echet topology . In algebraic setting, one need to use the $\mathfrak{m}$-adic topology on $\F$. However, the associated groups and spaces in question remain the same, though the topologies on them will be different. Since our arguments work for both Fr\'echet and $\mathfrak{m}$-adic settings, the topology will not be mentioned explicitly unless necessary. We also use $\widehat{V} = \Spf \F$ to denote the formal polydisc, either as a formal analytic space or a formal scheme.

We recall several definitions in \S 4, \cite{Diagonal}:
\begin{equation}\label{eq:progroup}
  \begin{split}
    G\formal(V) 
      &= \text{the proalgebraic group of automorphisms of the formal space } \widehat{V}, \\
    J\formal(V) 
      &= \Ker (d_0: G\formal \to \GL_n(V)), \text{ where } d_0(\phi) \text{ is the tangent map of } \phi \in G\formal \text{ at } 0,  \\
    \mathbf{g}\formal(V) 
      &= \text{Lie algebra of } G\formal(V)  \\
      &= \text{Lie algebra of formal vector fields on } \widehat{V} \text{ vanishing at } 0  \\
      &= \prod_{i \geq 1} V \otimes S^i V^* \\
      &= \prod_{i \geq 1} \Hom (V^*, S^i V^*),  \\
    \mathbf{j}\formal(V) 
      &= \text{Lie algebra of } J\formal(V)  \\
      &= \text{Formal vector fields on } \widehat{V} \text{ with vanishing constant and linear terms} \\
      &= \prod_{i \geq 2} V \otimes S^i V^* \\
      &= \prod_{i \geq 2} \Hom (V^*, S^i V^*),   
  \end{split}
\end{equation}
$\mathbf{g}\formal(V)$ and $\mathbf{j}\formal(V)$ act on $\F = \prod_{i \geq 0} S^i V^*$ as derivations in the obvious way. For convenience we will write $G\formal = G\formal(V)$ and so on.

The short exact sequence
\begin{equation}\label{exsq:progroups}
  1 \to J\formal \to G\formal \to \GL_n(V) \to 1
\end{equation}
splits canonically by identifying elements of $\GL_n = \GL_n(V)$ as jets of linear transformations on $\widehat{V}$. Hence $G\formal$ is the semidirect product $G\formal = J\formal \rtimes \GL_n$. The canonical bijection of sets
\begin{equation}\label{eq:J=G/GL}
  q: J\formal \xrightarrow{\simeq} G\formal / \GL_n
\end{equation}
is $G\formal$-equivariant if we endow $G\formal / \GL_n$ with the usual left $G\formal$-action and the left $G\formal$-action on $J\formal$ is given by
\begin{equation}\label{eq:actionJ_2}
  \psi \cdot \varphi = \psi \circ \varphi \circ (d_0 \psi)^{-1}, \quad \forall ~ \psi \in G\formal, ~ \forall ~ \varphi \in J\formal.
\end{equation}
In view of \eqref{eq:actionJ_2}, it is natural to interpret $J\formal$ as the set of \emph{formal exponential map} $\varphi: \widehat{T_0 V} \xrightarrow{\simeq} \widehat{V}$, where $\widehat{T_0 V}$ is the completion of the tangent space $T_0 V$ at the origin, which is canonically identified with $\widehat{V}$ itself. Such $\varphi$ induces the identity map on the tangent spaces of the two formal spaces at the origins. Moreover, giving a formal exponential map $\varphi$ is equivalent to giving an isomorphism between filtered algebras $\varphi^*: \F \to \F_T$, where
  \[  \F_T := \Shat ((T_0 V)^*) = \Shat (V^*) \]
is the algebra of functions on $\widehat{T_0 V}$, such that the induced isomorphism between the associated graded algebras, which are both $S(V^*)$, is the identity map.

In \cite{Diagonal} we introduced the set $\Conn$ of all flat torsion-free connections on $\widehat{V}$, i.e., each element of $\Conn$ is a (nonlinear) map
\begin{displaymath}
  \nabla : T\widehat{V} \to T\widehat{V} \otimes_{\Osheaf_{\widehat{V}}} T^*\widehat{V}
\end{displaymath}
satisfying the Leibniz rule and the flatness and torsion-freeness conditions. By abuse of notation, we also use $\nabla$ to denote the induced connection on the cotangent bundle of $\widehat{V}$ and its associated tensor bundles:
\begin{displaymath}
  \nabla : T^*\widehat{V} \to T^*\widehat{V} \otimes_{\Osheaf_{\widehat{V}}} T^*\widehat{V}.
\end{displaymath}
There is a canonical bijection
\begin{equation}\label{eq:Conn=J}
  \expmap : \Conn \xrightarrow{\simeq} J\formal, \quad \nabla \mapsto \exp_\nabla,
\end{equation}
which assigns to each connection $\nabla$ a formal exponential map $\expmap_\nabla : \widehat{T_0 V} \to \widehat{V}$, which is completely determined by the way it pulls back functions $f \in \F$,
\begin{equation}\label{eq:exp_definition}
  \expmap^*_\nabla (f) = (\nabla^i f |_0)_{i \geq 0} = (f(0), \nabla f |_0, \nabla^2 f |_0, \cdots) \in \prod_{i \geq 0} S^i (T_0 V)^* = \F_T,
\end{equation}
where $\nabla f = df$ and $\nabla^i f = \nabla^{i-1} df$ for $i \geq 2$.  The torsionfreeness and flatness of $\nabla$ guanrantee that the terms in the expression are symmetric tensors.

Moreover, $G\formal$ naturally acts on $\Conn$ from left by pushing forward connections via automorphisms of $\widehat{V}$. By Lemma 4.1, \cite{Diagonal}, the bijection $\exp : \Conn \to J\formal$ is $G\formal$-equivariant.

\subsubsection{Bundle of formal coordinates and connections}\label{subsec:Formal_Geometry}

We introduce the bundle of formal coordinate systems $X_{coord}$ of a smooth complex manifold $X$ from \S 4.4., \cite{Kapranov}. At each point $x \in X$ the fiber $X_{coord, x}$ is the space of infinite jets of biholomorphisms $\varphi: V \simeq \complex^n \to X$ with $\varphi(0) = x$. $X_{coord}$ is naturally a holomorphic principal $G\formal$-bundle.

We can apply the associated bundle construction to the principal $G\formal$-bundle $X_{coord}$ to globalize various objects defined in \S \ref{subsec:Formal_Discs}. There is a canonical isomorphism between bundles of algebras
\begin{displaymath}
  X_{coord} \times_{G\formal} \F \simeq \mathcal{J}^\infty_X
\end{displaymath}
and hence we have a tautological trivalization of the jet bundle $\JetX$ over $X_{coord}$
\begin{equation}\label{eq:coord}
  X_{coord} \times_X \mathcal{J}^\infty_X \simeq X_{coord} \times \F.
\end{equation}
Other related jet bundles, such as $\mathcal{J}^\infty T_X$ ($\mathcal{J}^\infty T^*_X$, resp.), the jet bundle of the tangent bundle (cotangent bundle, resp.), can be obtained in a similar way by the associated bundle construction.

Another related bundle $\pi: X_{exp} \to X$ is the bundle of formal exponential maps introduced in \cite{Kapranov}, which we denote by $X_{exp}$. At each $x \in X$ the fiber $X_{exp, x}$  is the space of jets of holomorphic maps $\phi: T_x X \to X$ such that $\phi(0) = x$, $d_0 \phi = \Id_{T_x X}$. We have a canonical isomorphism
\begin{displaymath}
  X_{exp} \to X_{coord} \times_{G\formal} J\formal \simeq X_{coord} \times_{G\formal} (G\formal / \GL_n)
\end{displaymath}
which hence induces a biholomorphism
\begin{equation}\label{eq:GL=exp}
  X_{coord} / \GL_n \simeq X_{exp}
\end{equation}
We also defined in \cite{Diagonal} the bundle of jets of flat torsion-free connection
\begin{displaymath}
  X_{conn} = X_{coord} \times_{G\formal} \Conn
\end{displaymath}
whose fiber at a each point $x \in X$ consists of all flat torsion-free connections on the formal neighborhood of $x$. The $G\formal$-equivariant bijection  $\expmap : \Conn \xrightarrow{\simeq} J\formal$ induces an identification between the $X_{conn}$ and $X_{exp}$. We regard them as the same bundle with different descriptions.

There is a tautological flat and torsion-free connection over $X_{conn}$,
\begin{displaymath}
  \nabla_{tau} : \pi^* \mathcal{J}^\infty T^*X \to \pi^* \mathcal{J}^\infty T^*X \otimes_{\pi^* \mathcal{J}^\infty_X} \pi^* \mathcal{J}^\infty T^*X,
\end{displaymath}
which is $\Osheaf_{X_{conn}}$-linear and satisfies the Leibniz rule with respect to the differential
\begin{displaymath}
  \widetilde{d}\formal: \pi^* \mathcal{J}^\infty_X \to \pi^* \mathcal{J}^\infty T^*X
\end{displaymath}
that is the pullback of
\begin{displaymath}
  d\formal: \mathcal{J}^\infty_X \to \mathcal{J}^\infty T^*X.
\end{displaymath}
Here $d\formal$ is a $\Osheaf_X$-linear differential obtained by apply the associated bundle construction with $X_{coord}$ and the de Rham differential $d : \Osheaf_{\widehat{V}} \to T^*\widehat{V}$ on the formal disc $\widehat{V}$.

On the other hand, since $X_{conn}$ can also be interpreted as the bundle $X_{exp}$ of formal exponential maps, we have a tautological isomorphism between sheaves of algebras over $X_{conn}=X_{exp}$,
\begin{displaymath}
  Exp^*: \pi^* (X_{coord} \times_{G\formal} \F) \to \pi^* (X_{coord} \times_{G\formal} \F_T).
\end{displaymath}
The domain is identified as $\pi^* \JetX$ or $\pi^* \Osheaf_{\Xdiag}$, the pullback via $\pi$ of the structure sheaf of the formal neighborhood of the diagonal in $X \times X$, while for the codomain we have
\begin{displaymath}
  X_{coord} \times_{G\formal} \F_T \simeq X_{coord} \times_{G\formal} \GL_n \times_{\GL_n} \F_T \simeq X_{coord} / J\formal \times_{\GL_n} \F_T
\end{displaymath}
by our definition of the $G\formal$-action on $\F_T$. But the principal $\GL_n$-bundle $X_{coord} / J\formal$ is exactly the bundle of ($0$th-order) frames on $X$, so
\begin{displaymath}
  X_{coord} / J\formal \times_{\GL_n} V  \simeq   X_{coord} / J\formal \times_{\GL_n} T_0 V   \simeq TX.
\end{displaymath}
and similarly 
\begin{displaymath}
  X_{coord} / J\formal \times_{\GL_n} V^*   \simeq T^*X.
\end{displaymath}
Since the $\GL_n$-action respects the decomposition $\F_T = \prod_{i \geq 0} S^i V^*$, we get
\begin{displaymath}
  X_{coord} / J\formal \times_{\GL_n} \F_T \simeq \prod_{i \geq 0} S^i T^*X = \Shat(T^*X),
\end{displaymath}
which is the structure sheaf of $X\formal_{TX}$, the formal neighborhood of the zero section of $TX$. In short, we have a \emph{tautological exponential map}
\begin{displaymath}
  Exp: \pi^* X\formal_{TX} \to \pi^* \Xdiag
\end{displaymath}
or equivalently, a  \emph{tautological Taylor expansion map}
\begin{displaymath}
  Exp^*: \pi^* \Osheaf_{\Xdiag} \to \pi^* \Osheaf_{X\formal_{TX}}.
\end{displaymath}
which is an isomorphism of bundles of topological algebras. The induced map between associated bundle of graded algebras
\begin{displaymath}
  \graded Exp^* : \pi^* \graded \Osheaf_{\Xdiag} = \pi^* S(T^*X) \to \pi^* S(T^*X)
\end{displaymath}
is the identity map. By \eqref{eq:exp_definition} we can write $Exp^*$ explicitly in terms of $\nabla_{tau}$,
\begin{equation}
  Exp^* (f) = (\nabla_{tau}^i f |_0)_{i \geq 0} = (f(0), \nabla_{tau} f |_0, \nabla^2_{tau} f |_0, \cdots) \in \pi^* \Shat(T^*X),
\end{equation}
where $|_0$ stands for the 'restriction to the origin' map $\pi^* S^i \mathcal{J}^\infty T^*X \to \pi^* S^i T^*X$. It is the globalization of the natural restriction map $T^*\widehat{V} \to T^*_0 \widehat{V} = V^*$ by applying the associated bundle construction with $X_{coord}$ and then pulling back onto $X_{conn}$ via $\pi$. Again $\nabla_{tau} f$ means $d\formal f$ and so on.

The Taylor expansion map $Exp^*$ induces a natural bijection between global smooth sections of $X_{conn}$ and all possible smooth isomorphisms between $\JetX$ and $\Shat(T^*X)$ which are the identity map on the level of associated graded algebras. Given any smooth section $\sigma$ of $X_{conn}$, we denote by
\begin{displaymath}
  \expmap^*_{\sigma} : \JetX \to \Shat(T^*X)
\end{displaymath}
the corresponding smooth homomorphism of bundles of algebras over $X$. It is holomorphic if and only if $\sigma$ is holomorphic. In general, $X_{conn}$ carries a flat $(0,1)$-connection $\overline{d}$, such that for any given smooth section $\sigma$ of $X_{conn}$, its anti-holomorphic differential
\begin{equation}\label{eq:omega}
  \omega_\sigma := \overline{d} \sigma \in \A^{0,1}(\mathbf{j}\formal(TX))
\end{equation}
is well-defined and it satisfies a Maurer-Cartan type equation
\begin{equation}\label{eq:MCeqn_omega}
  \dbar \omega_\sigma + \frac{1}{2} [\omega_\sigma, \omega_\sigma] = 0.
\end{equation}
We denote by 
\begin{displaymath}
  \alpha_\sigma^n \in \A^{0,1}_X(\Hom(S^n TX, TX)) = \A^{0,1}_X(\Hom(T^*X, S^n T^*X)).
\end{displaymath}
the $n$-th graded component of $\omega$ in the decomposition \eqref{eq:progroup}. By abuse of notation, we also denote the $\AOD(X)$-linear extension of $\exp^*_\sigma$ by
\begin{equation}
  \expmap^*_\sigma : \A^\bullet(X\formal_{X \times X}) \to \AOD_X(\Shat(T^*X))
\end{equation}
between graded algebras. It is in general not a homomorphism of dgas. The deficiency is measured exactly by $\omega$ since
\begin{equation}\label{eq:omega_exp}
  \omega_\sigma = \dbar \expmap^*_\sigma \circ (\expmap^*_\sigma)^{-1},
\end{equation}
where $\dbar \expmap^*_\sigma = [\dbar, \expmap^*_\sigma]$. Define a new differential $D_\sigma = \dbar - \sum_{n \geq 2} \widetilde{\alpha}_\sigma^n$ on $\AOD_X(\Shat(T^*X)$, where $\widetilde{\alpha}_\sigma^n$ is the odd derivation of the graded algebra $\AOD(\Shat(T^*X))$ induced by $\alpha_\sigma^n$. Then $D_\sigma^2 = 0$ by \eqref{eq:MCeqn_omega} and we have

\begin{proposition}[\cite{Diagonal}, Prop. 4.4.]\label{prop:exponential_map}
For any given smooth section $\sigma$ of $X_{conn}$, the map
\begin{displaymath}
  \expmap^*_\sigma : (\A^\bullet(X\formal_{X \times X}),\dbar) \to (\AOD(\Shat(T^*X)), D_\sigma)
\end{displaymath}
is an isomorphism of dgas. 
\end{proposition}

\subsection{Kapranov's result for  K\"ahler manifolds}\label{subsec:Kapranov_Kahler}

Now suppose that $X$ is equipped with a K\"ahler metric $h$. Let $\nabla$ be the canonical $(1,0)$-connection in $TX$ associated with $h$, so that
\begin{equation}\label{eq:flatness}
  [\nabla,\nabla] = 0 ~\text{in}~ \A^{2,0}_X(\End(TX)).
\end{equation}
and it is torsion-free, which is equivalent to the condition for $h$ to be K\"ahler.

Set $\widetilde{\nabla} = \nabla + \dbar$, where $\dbar$ is the $(0,1)$-connection defining the complex structure. The curvature of $\widetilde{\nabla}$ is just
\begin{equation}\label{defn:AtiyahClass}
  R = [\dbar, \nabla] \in \A^{1,1}_X(\End(TX)) = \A^{0,1}_X(\Hom(TX \otimes TX, TX))
\end{equation}
which is a Dolbeault representative of the Atiyah class $\alpha_{TX}$ of the tangent bundle. In particular one has the Bianchi identity:
\begin{displaymath}
  \dbar R = 0 ~ \text{in} ~ \A^{0,2}_X(\Hom(TX \otimes TX, TX))
\end{displaymath}
Actually, by the torsion-freeness we have
\begin{displaymath}
  R \in \A^{0,1}_X(\Hom(S^2TX,TX))
\end{displaymath}
Now define tensor fields $R_n$, $n \geq 2$, as higher covariant derivatives of the curvature:
\begin{equation}\label{defn:Derivative_AtiyahClass}
  R_n \in \A^{0,1}_X(\Hom(S^2TX \otimes TX^{\otimes(n-2)}, TX)), \quad R_2:=R, \quad R_{i+1}=\nabla R_i
\end{equation}
In fact $R_n$ is totally symmetric, i.e.,
\begin{displaymath}
  R_n \in \A^{0,1}_X(\Hom(S^n TX,TX)) = \A^{0,1}_X(\Hom(T^*X, S^n T^*X))
\end{displaymath}
by the flatness of $\nabla$ \eqref{eq:flatness}. Note that if we think of $\nabla$ as the induced connection on the cotangent bundle, the same formulas \eqref{defn:AtiyahClass} and \eqref{defn:Derivative_AtiyahClass} give $-R_n$.

The connection $\nabla$ determines a smooth section of $X_{conn}$,which we write as $\sigma=[\nabla]_\infty$, by assigning to each point $x \in X$ the holomorphic jets of $\nabla$. This has been done implicitly in the proof of Lemma 2.9.1., \cite{Kapranov}. One can check that the induced Taylor expansion map
\begin{displaymath}
  \expmap^*_\sigma : \A^\bullet(\Xdiag) \xrightarrow{\simeq} \AOD_X(\Shat(T^*X))
\end{displaymath}
by
\begin{equation}\label{eq:exp_Kahler}
  \expmap^*_\sigma ([\eta]_\infty) = (\Delta^*\eta, \Delta^*\nabla \eta, \Delta^*\nabla^2 \eta, \cdots, \Delta^*\nabla^n \eta, \cdots) \in \AOD_X(\Shat(T^*X))
\end{equation}
for any $[\eta]_\infty \in \A^\bullet(\Xdiag)$. Here $\nabla$ is understood as the pullback of $\nabla$ (on the cotangent bundle) via $\proj_2$, which is a constant family of connections along fibers of $\proj_1$, instead of jets of $\nabla$. The key observation is that the right hand side of the formula \label{eq:exp_Kahler} only depends on the class $[\eta]_\infty \in \A^\bullet(\Xdiag)$ and the holomorphic jets of the $\nabla$. In \cite{Diagonal}, we reformulated and reproved Theorem 2.8.2., \cite{Kapranov} in our language.

\begin{theorem}[Thm 3.2., \cite{Diagonal}]\label{thm:Kapranov_Diagonal}
  Assume $X$ is K\"ahler. With the notations from \S \ref{subsec:Formal_Geometry} and above, we have $\alpha_\sigma^n = - R_n$, i.e., $\omega_\sigma = - \sum_{n \geq 2} R_n$. Thus there is an isomorphism between dgas
  \begin{displaymath}
    \expmap^*_\sigma : (\A^\bullet(\Xdiag),\dbar) \to (\AOD_X(\Shat(T^*X)), D_\sigma)
  \end{displaymath}
  The derivation $D_\sigma = \dbar + \sum_{n \geq 2} \widetilde{R}_n$, where $\widetilde{R}_n$ is the odd derivation of $\AOD(\Shat(T^*X))$ induced by $R_n$.
\end{theorem}

We conclude this section by a slightly more generalized version of Theorem \ref{thm:Kapranov_Diagonal}, which will be used later in \S ~ \ref{subsec:Taylor_normal}. Suppose $f: X \to Y$ is an arbitrary holomorphic map. We consider the graph of $f$
\begin{displaymath}
\widetilde{f} = (\Id,f): X \to X \times Y
\end{displaymath}
which is a closed embedding. So we can consider the formal neighborhood $X_{X \times Y}\formal$. All the constructions above can be carried out in almost the same way with only slight adjustment and give us a description of the Dolbeault dga $\A^\bullet(X_{X \times Y}\formal)$. Namely, consider the pullback bundle $f^* Y_{conn}$ over $X$. Each smooth section $\sigma$ of $f^* Y_{conn}$ naturally corresponds to an isomorphism
\begin{displaymath}
  \eta_\sigma : \A^\bullet(X_{X \times Y}\formal) \hookrightarrow \AOD_X(\Shat(f^*T^*Y))
\end{displaymath}
of graded algebras. One can also view sections of $f^*Y_{conn}$ as jets of flat torsion-free connections on $X\formal_{X \times Y}$ along $Y$-fibers. In particular, when $Y$ carries a K\"ahler metric and the canonical $(1,0)$-connection $\nabla$, we can pullback $\nabla$ via the projection $X \times Y \to Y$, which determines a smooth section $\sigma$ of $f^* Y_{conn}$ and a Taylor expansion map
\begin{displaymath}
  \expmap^*_\sigma : \A^\bullet(X_{X \times Y}\formal) \to \AOD_X(\Shat(f^* T^*Y)), \quad [\zeta]_\infty \mapsto (\widetilde{f}^*\zeta, \widetilde{f}^*\nabla \zeta, \widetilde{f}^*\nabla^2 \zeta, \cdots),
\end{displaymath}
for any $[\zeta]_\infty \in \A^\bullet(X_{X \times Y}\formal)$. By abuse of notations, we still write
\begin{displaymath}
  R_n \in \A^{0,1}_Z(\Hom( f^*TY, S^n(f^*TY)))
\end{displaymath}
as the pullback of the curvature form of $Y$ and its covariant derivatives via $f$. Then we have the following theorem,

\begin{theorem}\label{thm:formal_graph}
  We have an isomorphism between dgas
  \begin{displaymath}
    \expmap^*_\sigma : (\A^\bullet(X_{X \times Y}\formal),\dbar) \xrightarrow{\simeq} (\AOD_X(\Shat^\bullet(f^* T^*Y)),D_\sigma)
  \end{displaymath}
  where $D_\sigma = \dbar + \sum_{n \geq 2} \widetilde{R}_n$ and $\widetilde{R}_n$ is the derivation of degree $+1$ induced by $R_n$.
\end{theorem}

\section{Case of general embeddings}\label{sec:General_Case}

Let $i: X \hookrightarrow Y$ be an arbitrary embedding and $\A^\bullet(\Yhat)$ the Dolbeault dga associated to the formal neighborhood of $X$ inside $Y$ as in \S~\ref{subsec:Formal_Geometry}. The goal is to build some appropriate isomorphism
$\A^\bullet(\Yhat) \simeq \AOD_X(\Shat^\bullet(\conormal))$ and write down the $\dbar$-derivation explicitly under this identification. We will show that this can be derived from the special yet universal case considered in \S ~\ref{sec:Diagonal_Formal}.

\subsection{Differential geometry of complex submanifolds}

\subsubsection{Splitting of normal exact sequence and Kodaira-Spencer class}\label{subsubsec:KS_class}

Over $X$ we have the short exact sequence of holomorphic vector bundles defining the normal bundle $N=\normal$
\begin{equation}\label{seq:normal}
  0 \to TX \xrightarrow{\iota} i^*TY \xrightarrow{p} N \to 0
\end{equation}
and its dual
\begin{equation}\label{seq:conormal}
  0 \to \conormal \xrightarrow{p^\vee} i^*T^*Y \xrightarrow{\iota^\vee} T^*X \to 0
\end{equation}
We fix a choice of \Cinf-splitting of the normal exact sequence \eqref{seq:normal}, i.e., two smooth homomorphisms of vector bundles $\projtan: i^*TY \to TX$ and $\liftnormal: N \to i^*TY$ satisfying
\begin{displaymath}
  \projtan \circ \iota = \Id_{TX}, \quad p \circ \liftnormal = \Id_{N}, \quad
  \iota \circ \projtan + \liftnormal \circ p = \Id_{i^*TY}
\end{displaymath}
and denote the corresponding dual splitting on the conormal exact sequence \eqref{seq:conormal} by $\projtan^\vee: T^*X \to i^*T^*Y$ and $\liftnormal^\vee: i^*T^*Y \to \conormal$. We can choose the splittings as the orthonormal decomposition induced by a K\"ahler metric on $Y$ if there is one, but again we will never need a metric in our general discussion.

Think of $\projtan^\vee$ as a \Cinf-section of the holomorphic vector bundle $\Hom(T^*X,i^*T^*Y),$ we can form
  \[ \beta= \beta_{X/Y} := \dbar \projtan^\vee \in \A^{0,1}_X(\Hom(T^*X,T^*Y)). \]
In fact
\begin{displaymath}
  \beta \in \A^{0,1}_X(\Hom(T^*X,\conormal)).
\end{displaymath}
To see this, just apply $\dbar$ on both sides of equality $\iota^\vee \circ \projtan^\vee = \Id_{T^*X}$ and note that $\iota^\vee$ is holomorphic and hence $\dbar \iota^\vee = 0$. By definition $ \dbar \beta = 0$, thus $\beta$ defines a cohomology class $[\beta] \in \gext^1_X(T^*X,\conormal)$, which is the obstruction class for the existence of a holomorphic splitting of the exact sequence \eqref{seq:conormal} or \eqref{seq:normal}. We call it the \emph{Kodaira-Spencer class}. Also note that
\begin{equation}\label{eq:beta_liftnormal}
\beta = - \dbar \liftnormal = - \dbar \liftnormal^\vee \in \A^{0,1}(\Hom(N,TX)) = \A^{0,1}(\Hom(T^*X,\conormal))
\end{equation}

\subsubsection{Shape operator}\label{subsubsec:diff_geom}

Suppose $\nabla$ is an arbitrary $(1,0)$-connection on $TY$ without any additional assumption. We use the same notation for the pullback connection on $i^*TY$ or $i^*T^*Y$ over $X$. The induced connection on the normal bundle via the chosen splitting is denoted by $\normalconn$:
\begin{displaymath}
  \normalconn_V \mu := p(\nabla_V \liftnormal(\mu)) \in C^\infty_X(N), \quad \forall~ \mu \in C^\infty_X(N), ~ V \in C^\infty(TX)
\end{displaymath}
(here we identify $T^{1,0}X$ with $TX$). Analogous to the shape operator in Riemannian geometry, we also define a linear operator $\shape : TX \otimes N \to TX$ by
\begin{equation}\label{eq:shape_operator}
  \shape^\mu(V) = - \projtan(\nabla_V \liftnormal(\mu)), \quad \forall~ \mu \in C^\infty_X(N), ~ V \in C^\infty(TX).
\end{equation}
That is, we first lift a smooth section $\mu$ of the normal bundle to a section of $i^*TY$ via the splitting, then take its derivatives with respect to the induced connection on $i^*TY$ and finally project the output into $TX$. Note that $\shape$ is in general not a holomorphic map between vector bundles.

\subsection{Taylor expansions in normal direction}\label{subsec:Taylor_normal}

\subsubsection{General discussions}

Let $i: X \hookrightarrow Y$ be a closed embedding, where $Y$ is not necessarily K\"ahler. Similar to what has been done in \S~ \ref{subsec:Formal_Geometry}, we can consider all isomorphisms
\begin{displaymath}
  \A^\bullet(\Yhat) \xrightarrow{\simeq} \AOD_X(\Shat(\conormal))
\end{displaymath}
which induces identity on the associated graded bundle $S(\conormal)$ and there is a infinite dimensional bundle $\Psi_{X/Y} \to X$ whose smooth sections correspond exactly to such isomorphisms. Indeed, for each $x \in X$, the fiber $\Psi_{X/Y, x}$ is the space of jets of holomorphic maps $\psi: N_{x} \to Y$ with $\psi(0) = x$ and $p \circ d_0 \psi = \Id_{N_x}$, where $N_x$ is the fiber of the normal bundle at point $x \in X$ and $p : i^*TY \to N$ is the natural projection.

Similarly, we can define another bundle $\Theta_{X/Y}$ over $X$ whose fiber at $x$ is the space of jets of maps $\theta: N_{x} \to T_x Y$ with $\theta(0) = x$ and $p \circ d_0 \theta = \Id_{N_x}$. Then $\Theta_{X/Y}$ admits a natural left action of $J\formal(T^*Y)$ (or more precisely, $J\formal(T^*Y|_X)$). In fact, we have a canonical isomorphism
\begin{displaymath}
  \Psi_{X/Y} \simeq Y_{exp}|_X \times_{J\formal(T^*Y)} \Theta_{X/Y} \simeq Y_{conn}|_X \times_{J\formal(T^*Y)} \Theta_{X/Y}.
\end{displaymath}

For our purpose here, however, it is not convenient to deal with $\Psi_{X/Y}$ since even we have already understood $Y_{conn}$ in various geometric ways, general sections of the bundle $\Theta_{X/Y}$ are difficult to handle. So instead we only look at linear liftings $N_x \to TY_x$ which form a subbundle $\Theta_{X/Y}^{(1)} \subset \Theta_{X/Y}$. Moreover, there is a canonical retraction $\Theta_{X/Y} \to \Theta_{X/Y}^{(1)}$ sending jets of maps $\psi: N_{x} \hookrightarrow Y$ to their linearizations $d_0 \psi$. Thus we have a fiberwise surjection of bundles
\begin{displaymath}
  \kappa : Y_{conn}|_X \times_X \Theta_{X/Y}^{(1)} \to Y_{conn}|_X \times_{J\formal(T^*Y)} \Theta_{X/Y} = \Psi_{X/Y}
\end{displaymath}
As an affine bundle over the vector bundle $\conormal \otimes TX$, $\Theta_{X/Y}^{(1)}$ admits smooth sections which are nothing but \Cinf-liftings $\rho: N \to TY$. Such a lifting $\rho$ and any section $\sigma$ of $Y_{conn}|_X$ together determine a section $\Xi$ of $\Psi_{X/Y}$ via $\kappa$, and hence an isomorphism of graded algebras
\begin{displaymath}
  \expmap^*_{X/Y, \Xi} : \A^\bullet(\Yhat) \xrightarrow{\simeq} \AOD_X(\Shat^\bullet(\conormal))
\end{displaymath}
The rest of the job is to determine which differential we should put on the codomain of this isomorphism in terms of $\rho$ and $\sigma$ to make it into an isomorphism of dgas.

As in Theorem \ref{thm:formal_graph}, denote the graph of $i$ by
\begin{displaymath}
  \widetilde{i} := (\Id, i) : X \to X \times Y.
\end{displaymath}
We regard $X$ as a submanifold of $X \times Y$ via $\widetilde{i}$, then by Theorem \ref{thm:formal_graph} a section $\sigma$ of $Y_{conn}|_X$ induces
\begin{displaymath}
  \expmap^*_\sigma : (\A^\bullet(X_{X \times Y}\formal),\dbar) \xrightarrow{\simeq} (\AOD_X(\Shat^\bullet(T^*Y)),D_\sigma)
\end{displaymath}
where $T^*Y$ is understood as the pullback $i^*T^*Y$ (we will omit $i^*$ as long as there is no confusion). The derivation
\begin{equation}\label{eq:deriv_Y_times_X}
D_\sigma = \dbar + \sum_{n \geq 2} \widetilde{R}_n
\end{equation}
and $\widetilde{R}_n$ is induced by (the pullback of) the covariant derivatives of curvature forms of $Y$.

Note that we have a commutative diagram of holomorphic maps
\begin{diagram}
  X        &\rInto^{\widetilde{i}}    & X \times Y   \\
  \dEqual  &                      &\dOnto_{\pi}  \\
  X        &\rInto^{i}            & Y
\end{diagram}
where $\pi: X \times Y \to Y$ is the natural projection. Thus by functoriality, $\pi$ induces an injective homomorphism of dgas
\begin{equation}\label{eq:pistar}
  \pi^* : (\A^\bullet(\Yhat),\dbar) \to (\AOD(X_{X \times Y}\formal),\dbar), \quad [\eta]_\infty \mapsto [\pi^*\eta]_\infty.
\end{equation}
We then compose $\pi^*$ with the isomorphism $\expmap^*_\sigma$ to get
\begin{displaymath}
  \expmap^*_\sigma \circ ~\pi^* : (\A^\bullet(\Yhat),\dbar) \to (\AOD_X(\Shat^\bullet(T^*Y),D_\sigma)
\end{displaymath}

\begin{comment}
\begin{equation}\label{eq:exp_comp_pi}
  \expmap^*_\sigma \circ \pi^* : (\AOD(\Yhat),\dbar) \to (\AOD_X(\Shat^\bullet(T^*Y),D_\sigma), \quad [\eta]_\infty \mapsto (i^*\eta, i^* \nabla \eta, i^* \nabla^2 \eta, \cdots)
\end{equation}
since $\widetilde{i}^* \nabla^n \pi^* \eta = i^* \nabla^n \eta$, where the $\nabla$ on the RHS is the original $(1,0)$-connection on $Y$ while the one on the LHS is the pullback one along $Y$-fibers of $X \times Y$.

Finally we want to project $\Shat^\bullet(T^*Y)$ onto $\Shat^\bullet(\conormal)$. There is no natural way to accomplish this, however, so we need to use the \Cinf-splitting we have fixed.
\end{comment}

We then extend $\liftnormal^\vee: T^*Y \to \conormal$ to obtain a homomorphism of graded algebras
\begin{equation}\label{eq:projconormal}
  \liftnormal^\vee : \AOD_X(\Shat^\bullet(T^*Y)) \to \AOD_X(\Shat^\bullet(\conormal)).
\end{equation}
Composing with $\expmap^*_\sigma \circ \pi^*$ in \eqref{eq:exp_comp_pi} we get a homomorphism of graded algebras
\begin{displaymath}
  \liftnormal^\vee \circ \expmap^*_\sigma \circ \pi^* ~:~ \A^\bullet(\Yhat) \to \AOD_X(\Shat^\bullet(\conormal)).
\end{displaymath}
%which, in fact, coincides with $\expmap^*_{X/Y,\Xi}$.

\begin{lemma}
  With all the notations above, we have
  \begin{equation}\label{eq:expXY}
    \expmap^*_{X/Y,\Xi} = \liftnormal^\vee \circ \expmap^*_\sigma \circ \pi^*
  \end{equation}
\end{lemma}

\begin{proof}
  Follows immediately from the definition of $\kappa$.
\end{proof}

Via the isomorphism $\expmap^*_{X/Y,\Xi}$ we can transfer the $\dbar$-derivation on $\A^\bullet(\Yhat)$ to a derivation on $\AOD_X(\Shat^\bullet(\conormal))$, denoted as $\Dnormal$, i.e.,
\begin{displaymath}
  \Dnormal = \expmap^*_{X/Y,\Xi} \circ ~\dbar \circ (\expmap^*_{X/Y,\Xi})^{-1}.
\end{displaymath}
Hence $\expmap^*_{X/Y,\Xi}$ becomes an isomorphism of dgas
\begin{displaymath}
  \expmap^*_{X/Y,\Xi} : (\A^\bullet(\Yhat),\dbar) \xrightarrow{\simeq} (\AOD_X(\Shat^\bullet(\conormal)),\Dnormal).
\end{displaymath}
Thus we can also transfer the homomorphism $\pi^*$ in \eqref{eq:pistar} to a homomorphism
\begin{displaymath}
  \widetilde{\pi}^* : (\AOD_X(\Shat^\bullet(\conormal)),\Dnormal) \to (\AOD_X(\Shat^\bullet(T^*Y),D_\sigma),
\end{displaymath}
by setting
\begin{equation}\label{eq:tildepi}
\widetilde{\pi}^* = \expmap^*_\sigma \circ ~\pi^* \circ (\expmap^*_{X/Y,\Xi})^{-1}.
\end{equation}
We then get the following commutative diagram:

\begin{diagram}[width=6em,height=3em]
  (\A^\bullet(\Yhat),\dbar)                           &\rTo^{\pi^*}          &(\AOD(X_{X \times Y}\formal),\dbar)  \\
  \dTo_{\expmap^*_{X/Y,\Xi}}^{\simeq}               &                      &\dTo_{\expmap^*_\sigma}^{\simeq}      \\
  (\AOD_X(\Shat^\bullet(\conormal)),\Dnormal)   &\rTo{\widetilde{\pi}^*}   &(\AOD_X(\Shat^\bullet(T^*Y),D)    \\
\end{diagram}

Note that by \eqref{eq:expXY} and \eqref{eq:tildepi} we have
\begin{equation}\label{eq:retraction}
  \liftnormal^\vee \circ \widetilde{\pi}^* = \Id : (\AOD_X(\Shat^\bullet(\conormal)),\Dnormal) \to (\AOD_X(\Shat^\bullet(\conormal)),\Dnormal)
\end{equation}
even though $\liftnormal^\vee$ is not a homomorphism of dgas.

\subsubsection{Description of $\widetilde{\pi}^*$}\label{subsubsec:pi}

From now on, let us assume $Y$ is K\"ahler and the section $\sigma$ of $Y_{conn}|_X$ is determined by the associated $(1,0)$-connection $\nabla$. However, we want to keep the reader aware that all the arguments and computations below will still work even if the K\"ahler condition is dropped and the section $\sigma$ is arbitrary. All the $\nabla$ appearing in the formulae can be intepreted as formal connections determined by $\sigma$ without any change just as in Proposition \ref{prop:exponential_map}. The K\"ahler assumption just makes sure that those terms in the final formula \eqref{eq:Dnormal_Final} of $\Dnormal$ have clearer differential geometric meanings.

By the discussion at the end of \S \ref{subsec:Kapranov_Kahler}, the homomorphism
\begin{displaymath}
  \expmap^*_\sigma \circ ~\pi^* : (\A^\bullet(\Yhat),\dbar) \to (\AOD_X(\Shat^\bullet(T^*Y),D_\sigma)
\end{displaymath}
is given by
\begin{equation}\label{eq:exp_comp_pi}
  [\eta]_\infty \mapsto (i^*\eta, i^* \nabla \eta, i^* \nabla^2 \eta, \cdots)
\end{equation}
since $\widetilde{i}^* \nabla^n \pi^* \eta = i^* \nabla^n \eta$, where the $\nabla$ on the right hand side is the original $(1,0)$-connection on $Y$ while the one on the left hand side is the pullback one along $Y$-fibers of $X \times Y$.

Before we give a description of the homomorphism $\widetilde{\pi}^*$, we make some conventions on notations. We abuse the notations and write $TY = TX \oplus N$ and $T^*Y = T^*X \oplus \conormal$ induced by the fixed splittings. The decompositions extend to tensor products, i.e., tensor product $TY^{\otimes n}$ can be decomposed into direct sum of mixed tensors of $TX$ and $N$ components and similarly for $T^*Y^{\otimes n}$. The same happens for symmetric tensor products:
\begin{equation}\label{eq:symm_decomp}
  S^n T^*Y = \bigoplus_{p+q=n} S^p T^*X \cdot S^q \conormal,
\end{equation}
where the dot stands for the commutative multiplication in the symmetric algebras.
% S^n TY = \bigoplus_{p+q=n} S^p TX \cdot S^q \normal \quad \text{and} \quad 

We define a derivation
\begin{equation}\label{eq:defn_symmconn}
\symmconn : \AOD_X(\Shat^{\bullet} T^*Y) \to \AOD_X(\Shat^{\bullet+1} T^*Y)
\end{equation}
of degree $0$ for the grading from $\AOD_X$ as the composition of the operators
\begin{displaymath}
  \symmconn := \overline{\Sym} \circ \nabla^{TX}
\end{displaymath}
Namely, for any $\eta \in \AOD_X(S^n T^*Y)$, first apply the connection $\nabla^{TX}$ on $i^* S^n T^*Y$ induced by $\nabla$ to get a $T^*X \otimes S^n T^*Y$-valued form
\begin{displaymath}
\nabla^{TX} \eta \in \AOD_X(T^*X \otimes S^n T^*Y),
\end{displaymath}
then apply a variation of the usual symmetrization map
\begin{displaymath}
  \overline{\Sym} : T^*X \otimes S^n T^*Y \to S^{n+1} T^*Y,
\end{displaymath}
which we define on each component of the decomposition \eqref{eq:symm_decomp} as
\begin{displaymath}
  \overline{\Sym}_{m,n} : T^*X \otimes (S^{m-1} T^*X \cdot S^{n-m+1} \conormal) \to S^m T^*X \cdot S^{n-m+1} \conormal
\end{displaymath}
by the formula
\begin{multline}\label{eq:sym_operator}
  \overline{\Sym}_{m,n} (v_0 \otimes ( v_1 \cdot v_2 \cdots v_n )) = \frac{1}{m} v_0 \cdot v_1 \cdots v_n, \\
  \forall~ v_0 \otimes ( v_1 \cdot v_2 \cdots v_n ) \in T^*X \otimes (S^{m-1} T^*X \cdot S^{n-m+1} \conormal)
\end{multline}
and we finally get
\begin{displaymath}
  \symmconn \eta \in \AOD_X(S^{n+1} T^*Y)
\end{displaymath}
When $\eta \in \AOD_X(S^0(T^*Y)) = \AOD_X$, $\symmconn \eta = \projtan^\vee (\partial \eta) \in \AOD_X(T^*Y)$ where $\partial$ is the $(1,0)$-derivation of forms on $X$. We can inductively apply $\symmconn$ and get
\begin{displaymath}
  \symmconn^k : \AOD_X(\Shat^\bullet(T^*Y)) \to \AOD_X(\Shat^{\bullet+k}(T^*Y))
\end{displaymath}
In particular, since $\conormal$ is naturally identified as a subbundle of $T^*Y$ via $p^\vee : \conormal \to T^*Y$, we can form the restriction of $\symmconn^k$ on $\Shat^\bullet(\conormal)$:
\begin{displaymath}
\symmconn^k : \AOD_X(\Shat^\bullet(\conormal)) \to \AOD_X(\Shat^{\bullet + k}(T^*Y))
\end{displaymath}
Note that here $\symmconn^0$ is the natural inclusion $\Shat^\bullet(\conormal) \hookrightarrow \Shat^\bullet(T^*Y)$.

\begin{proposition}\label{prop:tildepi}
  We have
  \begin{displaymath}
    \widetilde{\pi}^* = \sum_{k=0}^\infty \symmconn^k
  \end{displaymath}
  where $\widetilde{\pi}^* : \AOD_X(\Shat^\bullet(\conormal)) \to \AOD_X(\Shat^\bullet(T^*Y))$ is as in \eqref{eq:tildepi}. That is, given $\mu = (\mu_k)_{k=0}^\infty \in \AOD_X(\Shat^\bullet(\conormal))$, the $n$-th component of its image $\nu = \widetilde{\pi}^*(\mu)$ is
  \begin{displaymath}
    \nu_n = \sum_{k=0}^n \symmconn^k \mu_{n-k} \in \AOD_X(S^n \conormal).
  \end{displaymath}
\end{proposition}

\begin{proof}
  Assume that $\expmap^*_{X/Y,\Xi}([\eta]_\infty) = \mu$ where $[\eta]_\infty \in \A^\bullet(\Yhat)$. By \eqref{eq:expXY} and \eqref{eq:exp_comp_pi}, this means
  \begin{displaymath}
    \mu_k = \liftnormal^\vee (i^* \nabla^k \eta)
  \end{displaymath}
  where $\liftnormal^\vee$ is the projection
  \begin{displaymath}
    \liftnormal^\vee : \AOD_X(\Shat^\bullet(T^*Y)) \to \AOD_X(\Shat^\bullet(\conormal))
  \end{displaymath}
  as in \eqref{eq:projconormal}. Moreover, by the definition of $\widetilde{\pi}^*$ \eqref{eq:tildepi}, we have
  \begin{displaymath}
    \widetilde{\pi}^*(\mu)
      = \expmap^*_\sigma \circ ~\pi^* \circ (\expmap^*_{X/Y,\Xi})^{-1} (\mu)
      = \expmap^*_\sigma \circ ~\pi^* ([\eta]_\infty)
      = (i^* \nabla^k \eta)_{k=0}^\infty.
  \end{displaymath}
  Thus all we need to show is that
  \begin{displaymath}
    i^* \nabla^n \eta = \sum_{k=0}^n \symmconn^k (\liftnormal^\vee (i^* \nabla^{n-k} \eta))
  \end{displaymath}
  for all $n \geq 0$. We prove it by induction on $n$. The $n=0$ case is trivial. For $n \geq 1$, note that we can write
  \begin{displaymath}
    i^* \nabla^n \eta = \liftnormal^\vee (i^* \nabla^n \eta) + \text{component in $T^*X \cdot S^{n-1} T^*Y \subset S^n T^*Y$}
  \end{displaymath}
  via the decomposition \eqref{eq:symm_decomp}. The second term on the right hand side is nothing but $\symmconn (i^* \nabla^{n-1} \eta)$. To see this, consider what happens when we evaluate $i^*\nabla^n \eta$ at some section $s$ of $TY^{\otimes n}$, which lies in a mixed tensor of $m$ copies of $TX$ ($m \geq 1$) and $n-m$ copies of $N$ (of arbitrary order). One can permute any of the $TX$-factors to the first place and plug it into the first $\nabla$ in $i^*\nabla^n \eta$, since $i^*\nabla^n \eta$ is a symmetric tensor. This implies that $i^*\nabla^n \eta$ should be a symmetrization of $\nabla^{TX}(i^*\nabla^{n-1} \eta)$. The value of the latter at $s$, however, is $m$ times what we need since $s$ contains $m$ $TX$-factors. This explains the fractional factor $1/m$ in the formula \eqref{eq:sym_operator}. Finally we end the proof by applying the inductive assumption.
\end{proof}

\begin{comment}
\begin{remark}
  One can further expand $\symmconn$ and $\symmconn^k$ in terms of second fundamental form and shape operator, which might be interesting. As we will see immediately, however, only a small part of the expansions will be relevant to the derivation $\Dnormal$, so we postpone the general discussion for the future.
\end{remark}
\end{comment}

\subsubsection{Description of the derivation $\Dnormal$}

To determine the derivation $\Dnormal$, note that by \eqref{eq:retraction} we have
\begin{displaymath}
  \Dnormal = \liftnormal^\vee \circ \widetilde{\pi}^* \circ \Dnormal = \liftnormal^\vee \circ D \circ \widetilde{\pi}^*
\end{displaymath}
where the last equality is by the definition of $\Dnormal$. Thus for given $\mu = (\mu_k)_{k=0}^\infty \in \AOD_X(\Shat^\bullet(\conormal))$, the $n$-th component of $\Dnormal\mu$ is
\begin{equation}\label{eq:Dnormal_V1}
  (\Dnormal\mu)_n = \sum_{s+t=n} \liftnormal^\vee \circ \dbar \circ \symmconn^s \mu_t + \sum_{\substack{r+s+t=n \\ r \geq 1}} \liftnormal^\vee \circ \widetilde{R}_{r+1} \circ \symmconn^s \mu_t
\end{equation}
by Proposition \ref{prop:tildepi} and \eqref{eq:deriv_Y_times_X}. To simplify the right hand side further, first observe that all we need are the components of $\symmconn^s \mu_t$ lying in $\Shat^\bullet(\conormal)$ and $T^*X \cdot \Shat^\bullet(\conormal)$ and we can ignore the remaining ones in $S^2 T^*X \cdot \Shat^\bullet(T^*Y)$. The reason is that if we apply the derivations $\dbar$ and $\widetilde{R}_n$ on any section from $S^2 T^*X \cdot \Shat^\bullet(T^*Y)$, the outcomes must lie in $T^*X \cdot \Shat^\bullet(T^*Y)$, which will then be eliminated by the projection $\liftnormal^\vee$.

Thus we first denote the projections onto the only two `effective' components respectively by
\begin{displaymath}
  P_0 = p^\vee \circ \liftnormal^\vee : \Shat^\bullet(T^*Y) \to \Shat^\bullet(\conormal) \subset \Shat^\bullet(T^*Y)
\end{displaymath}
and
\begin{displaymath}
  P_1 : \Shat^\bullet(T^*Y) \to T^*X \cdot \Shat^\bullet(\conormal) \subset \Shat^\bullet(T^*Y).
\end{displaymath}

Secondly, we define the derivation of degree $+1$
\begin{displaymath}
  \widetilde{\beta} : \AOD_X(\Shat^\bullet(T^*Y)) \to \A^{0,\bullet+1}_X(\Shat^{\bullet}(T^*Y))
\end{displaymath}
induced by
\begin{displaymath}
  \beta \in \A^{0,1}_X(\Hom(T^*X,\conormal)) \subset \A^{0,1}_X(\Hom(T^*Y,T^*Y)).
\end{displaymath}
where the last inclusion comes again from the splitting of $T^Y$. Note that $\widetilde{\beta}$ acts on $\AOD_X(\Shat^\bullet(\conormal)) \subset \AOD_X(\Shat^\bullet(T^*Y))$ as the zero map. So we can also think of $\widetilde{\beta}$ as the operator
\begin{displaymath}
  \widetilde{\beta} : \AOD_X(T^*X \cdot \Shat^\bullet(\conormal)) \to \A^{0,\bullet+1}_X(\Shat^{\bullet+1}(\conormal))
\end{displaymath}
The following lemma is immediate from \eqref{eq:beta_liftnormal}.

\begin{lemma}
  As derivations $\AOD_X(\Shat^\bullet(T^*Y)) \to \A^{0,\bullet+1}_X(\Shat^\bullet(\conormal))$,
  \begin{displaymath}
    [\liftnormal^\vee , \dbar] = \widetilde{\beta} \circ P_1
  \end{displaymath}
\end{lemma}

Hence the first sum on the right hand side of \eqref{eq:Dnormal_V1} can be rewritten as
\begin{equation}\label{eq:Dnormal_V1_1}
  \begin{split}
    \sum_{s+t=n} \liftnormal^\vee \circ \dbar \circ \symmconn^s \mu_t
    &= \sum_{s+t=n} \dbar \circ \liftnormal^\vee \circ \symmconn^s \mu_t + \sum_{s+t=n} \widetilde{\beta}_X \circ P_1 \circ \symmconn^s \mu_t  \\
    &= \dbar \mu_n + \sum_{s+t=n} \widetilde{\beta}_X \circ P_1 \circ \symmconn^s \mu_t
  \end{split}
\end{equation}
The last equality is because that $\symmconn^s \mu_t \in \AOD_X(T^*X \cdot \Shat^\bullet(T^*Y))$ unless $s=0$.

Thirdly, we split $R_n \in \A^{0,1}_X(\Hom(T^*Y, S^n T^*Y))$ into two components,
\begin{displaymath}
  R^\bot_n := \liftnormal^\vee \circ R_n \circ p^\vee \in \A^{0,1}_X(\Hom(\conormal, S^n \conormal))
\end{displaymath}
and
\begin{displaymath}
  R^\top_n := \liftnormal^\vee \circ R_n \circ \projtan^\vee \in \A^{0,1}_X(\Hom(T^*X, S^n \conormal))
\end{displaymath}
and denote the induced operators by
\begin{displaymath}
  \widetilde{R}^\bot_n : \AOD_X(\Shat^\bullet(\conormal)) \to \A^{0,\bullet+1}_X(\Shat^{\bullet+n-1}(\conormal))
\end{displaymath}
and
\begin{displaymath}
  \widetilde{R}^\top_n : \AOD_X(T^*X \cdot \Shat^\bullet(\conormal)) \to \A^{0,\bullet+1}_X(\Shat^{\bullet+n}(\conormal))
\end{displaymath}
respectively. To unify notations, we also write $R^\top_1 := \beta$ and $\widetilde{R}^\top_1 := \widetilde{\beta}$.

We can then split the second term on the right hand side of \eqref{eq:Dnormal_V1}:
\begin{equation}\label{eq:Dnormal_V1_2}
  \begin{split}
    \quad  &\sum_{\substack{r+s+t=n \\ r \geq 1}} \liftnormal^\vee \circ \widetilde{R}_{r+1} \circ \symmconn^s \mu_t  \\
    =      &\sum_{\substack{r+s+t=n \\ r \geq 1}} \liftnormal^\vee \circ \widetilde{R}_{r+1} \circ P_0 \circ \symmconn^s \mu_t
           + \sum_{\substack{r+s+t=n \\ r \geq 1}} \liftnormal^\vee \circ \widetilde{R}_{r+1} \circ P_1 \circ \symmconn^s \mu_t  \\
    =      &\sum_{\substack{r+s+t=n \\ r \geq 1}} \widetilde{R}^\bot_{r+1} \circ P_0 \circ \symmconn^s \mu_t
           + \sum_{\substack{r+s+t=n \\ r \geq 1}} \widetilde{R}^\top_{r+1} \circ P_1 \circ \symmconn^s \mu_t  \\
    =      &\sum_{k=2}^n \widetilde{R}^\bot_k \circ \mu_{n-k+1}
           + \sum_{\substack{r+s+t=n \\ r \geq 1}} \widetilde{R}^\top_{r+1} \circ P_1 \circ \symmconn^s \mu_t
  \end{split}
\end{equation}
Combine \eqref{eq:Dnormal_V1}, \eqref{eq:Dnormal_V1_1} and \eqref{eq:Dnormal_V1_2} we get
\begin{equation}\label{eq:Dnormal_V2}
  (\Dnormal \mu)_n
    = \dbar \mu_n
      + \sum_{k=2}^n \widetilde{R}^\bot_k \circ \mu_{n-k+1}
      + \sum_{\substack{r+s+t=n \\ r,t \geq 0, ~ s \geq 1}} \widetilde{R}^\top_{r+1} \circ P_1 \circ \symmconn^s \mu_t
\end{equation}

Finally, to compute the terms $P_1 \circ \symmconn^s \mu_t$, we define two derivations of degree $0$ with respect to the grading from $\AOD_X$:
\begin{displaymath}
  \normalconnbar : \AOD_X(\Shat^\bullet(\conormal)) \to \AOD_X(T^*X \cdot \Shat^\bullet(\conormal))
\end{displaymath}
induced by the connection $\normalconn$ on $\conormal$ the same way as we define $\symmconn$ in \eqref{eq:defn_symmconn}. $\normalconn f$ of a function $f$ is again understood as $\partial f$, the $(1,0)$ differential.

The second one
\begin{displaymath}
  \widetilde{\shape} : \AOD_X(T^*X \cdot \Shat^\bullet(\conormal)) \to \AOD_X(T^*X \cdot \Shat^{\bullet+1}(\conormal))
\end{displaymath}
induced by the shape operator $\shape : T^*X \to T^*X \otimes \conormal$ as in \eqref{eq:shape_operator} yet again with the images symmetrized.

\begin{lemma}
  With the notations above, we have
  \begin{displaymath}
    P_1 \circ \symmconn = \normalconnbar \circ P_0 + \widetilde{\shape} \circ P_1,
  \end{displaymath}
  thus $\forall~s \geq 1$, $\forall~\mu \in \Shat^\bullet(\conormal)$,
  \begin{equation}\label{eq:symmconn_power}
  P_1 \circ \symmconn^s \mu = (\widetilde{\shape})^{s-1} \circ \normalconnbar \mu.
  \end{equation}
\end{lemma}

Applying the equality \eqref{eq:symmconn_power} to \eqref{eq:Dnormal_V2}, we eventually get

\begin{theorem}\label{thm:MAIN_derivation}
  Given $\mu = (\mu_k)_{k=0}^\infty \in \AOD_X(\Shat^\bullet(\conormal))$, the $n$-th component of $\Dnormal\mu$ is
  \begin{displaymath}
    (\Dnormal \mu)_n
      = \dbar \mu_n
        + \sum_{k=2}^n \widetilde{R}^\bot_k \circ \mu_{n-k+1}
        + \sum_{\substack{r+s+t=n \\ r,t \geq 0, ~ s \geq 1}} \widetilde{R}^\top_{r+1} \circ \widetilde{\shape}^{s-1} \circ \normalconnbar \mu_t.
  \end{displaymath}
  In other words,
  \begin{equation}\label{eq:Dnormal_Final}
    \Dnormal = \dbar + \sum_{k \geq 2} \widetilde{R}^\bot_k + \sum_{p \geq 1,~ q \geq 0} \widetilde{R}^\top_p \circ \widetilde{\shape}^q \circ \normalconnbar.
  \end{equation}
\end{theorem}

\begin{remark}
  From Theorem \ref{thm:MAIN_derivation} we see that, even when $\Dnormal$ acts on a function $f$ (or a form), higher term would be produced in general. Namely, by \eqref{eq:Dnormal_Final}
  \begin{equation}\label{eq:Dnormal_zero}
    \Dnormal f = \dbar f + \sum_{p \geq 1,~ q \geq 0} \widetilde{R}^\top_p \circ \widetilde{\shape}^q \circ \partial f
  \end{equation}
  This is a huge difference between the general situation and the case of diagonal embedding.
\end{remark}

\begin{remark}
  Although we get $\Dnormal^2 = 0$ for free from how we construct it, it is still an interesting (yet tedious) exercise to verify it by hands and one will observe the Gauss-Codazzi-Ricci equations in classical differential geometry (see \cite{Xin}). We leave the details to interested readers.
\end{remark}

\section{$L_\infty$-algebroid of the formal neighborhood}\label{sec:sh-LieAlgebroid}

In this section we repackage the results in \S~\ref{sec:General_Case} in terms of a $L_\infty$-algebroid structure on the shifted cotangent bundle $N[-1]$, or more precisely, on the complex $\AOD_X(\normal[-1]) = \A^{0,\bullet-1}_X(\normal)$, whose Chevalley-Eilenberg complex is exactly the Dolbeault dga $\A^\bullet(Y\formal_X)$. However, to keep the signs simple, we will work with the \Linfs-algebroid structure on the unshifted normal bundle rather than the \Linf-algebroid. Since throughout this section the background algebra is the Dolbeault dga $(\AOD(X),\dbar)$ of the submanifold $X$, we will just write it as $\AOD$.

%Apology: I'm too lazy to nail down the minus signs...

\subsection{Conventions and notations}

We follow the notations and sign conventions in \cite{Luca}. For any postitive integers $k_1, \ldots, k_l$, let $\Sh(k_1,\ldots k_l)$ be the set of \emph{$(k_1, \ldots, k_l)$-unshuffles}, i.e., permutations $\sigma$ of set of integers $\{1, 2, \ldots, k_1+k_2 + \cdots + k_l\}$ satisfying
  \[ \sigma(k_1 + \cdots + k_{i-1} + m)  < \sigma(k_1 + \cdots + k_{i-1} + n), \quad \forall ~ 1 \leq m < n \leq k_i, ~ i = 1, \ldots, l.\]

Suppose $V = \oplus_i V^i$ is a graded vector space over a field $\field$ of zero characteristic. Given a list of homogeneous vectors in $\mathbf{v} = (v_1, \ldots, v_n)$ in $V$
and a permutation $\sigma \in S_n$, we denote by $\alpha(\sigma, \mathbf{v})$ (resp., $\chi(\sigma, \mathbf{v})$) the sign determined by 
  \[ v_{\sigma(1)} \odot \cdots \odot v_{\sigma(n)} = \alpha(\sigma,\mathbf{v}) v_1 \odot \cdots \odot v_n \quad (\text{resp.}~ v_{\sigma(1)} \wedge \cdots \wedge v_{\sigma(n)} = \chi(\sigma,\mathbf{v}) v_1 \wedge \cdots \wedge v_n) \]
where $\odot$ (resp., $\wedge$) is the graded symmetric (resp., graded skew-symmetric) product in the graded symmetric algebra $S(V)$ (resp., graded exterior algebra $\wedge V$).

\subsection{\Linf- and \Linfs-algebroids}
We recall the definition of \Linf-algebras and \Linfs-algebras (\cite{Luca}).

\begin{definition}\label{defn:shLie}
  An \Linf-algebra is a graded vector space $L^\bullet$ equipped with a family of $n$-ary multilinear operations ($n$-brackets)
  \begin{displaymath}
        l_n : (L^\bullet)^{\times n} \to L^{\bullet},\qquad (x_1 , \cdots , x_n) \mapsto [x_1,\ldots,x_n]_n, \quad n \in \nat
  \end{displaymath}
  of degree $2-n$, such that
  \begin{enumerate}
    \item
      $[\cdot,\ldots,\cdot]_n$ is graded skew-symmetric, i.e., for any permutation $\sigma \in S_n$ and vector $\mathbf{v}=(v_1, \ldots, v_n)$ of homogeneous elements in $L^\bullet$,
      \[ [v_{\sigma(1)}, \ldots, v_{\sigma(n)} ]_n = \chi(\sigma,\mathbf{v}) [v_1, \ldots, v_n]_n ,\]
      and
   \item
     the higher Jacobi identity is satisfied:
     \begin{equation}\label{eq:Jacobi}
         \sum_{i+j=n}\sum_{\sigma \in \Sh(i,j)} (-1)^{ij} \chi(\sigma, \mathbf{v}) [ [ v_{\sigma(1)}, \ldots v_{\sigma(i)}]_i, v_{\sigma(i+1)}, \ldots, v_{\sigma(n)}]_{j+1}= 0,
      \end{equation}
     for any $n \in \nat$. In particular, $(L^\bullet, d=[ ~\cdot~ ]_1)$ is a cochain complex.
      
  \end{enumerate}
\end{definition}

%It is convenient in our case to deal with the so-called \Linfs-algebras. One reason is the simplified signs. Another reason is that we can talk about $N$ instead of $N[-1]$.

\begin{definition}\label{defn:shLie}
  An \Linfs-algebra is a graded vector space $L^\bullet$ equipped with a family of $n$-ary multilinear operations ($n$-brackets)
      \[  \ell_n : (L^\bullet)^{\times n} \to L^{\bullet},\qquad (x_1 , \cdots , x_n) \mapsto \{x_1,\ldots,x_n\}_n, \quad n \in \nat  \]
  of degree $1$, such that
  \begin{enumerate}
    \item
      $\{ \cdot,\ldots,\cdot \}_n$ is graded symmetric, i.e., for any permutation $\sigma \in S_n$ and homogeneous vectors $\mathbf{v}=(v_1, \ldots, v_n)$ in $L^\bullet$,
      \[ \{ v_{\sigma(1)}, \ldots, v_{\sigma(n)} \}_n = \alpha(\sigma,\mathbf{v}) \{ v_1, \ldots, v_n \}_n ,\]
      and
   \item
     the higher Jacobi identity is satisfied:
     \begin{equation}\label{eq:Jacobi}
         \sum_{i+j=n}\sum_{\sigma \in \Sh(i,j)} \alpha(\sigma, \mathbf{v}) \{ \{ v_{\sigma(1)}, \ldots v_{\sigma(i)} \}_i, v_{\sigma(i+1)}, \ldots, v_{\sigma(n)} \}_{j+1}= 0,
      \end{equation}
     for any $n \in \nat$. In particular, $(L^\bullet, d=\{ ~\cdot~ \}_1)$ is a cochain complex.
  \end{enumerate}
  In the case when $l_n=0$ for all $n \geq 2$, we call $L^\bullet$ as a \emph{shifted differential graded Lie algebra} or simply a \emph{shifted DGLA}. 
\end{definition}
There is a one-to-one correspondence between \Linf-algebra structures $\{ [\cdot, \cdots, \cdot]_n | ~n \in \nat \}$ on a graded vector space $L$, and \Linfs-algebra structures $\{ \{\cdot, \cdots, \cdot\}_n |~n \in \nat \}$ on $L[1]$, the shifted graded vector space of $L$, given by
  \[ \{ v_1, \ldots, v_n  \} = (-1)^{(k-1)|v_1| + (k-2)|v_2|+ \cdots + |v_{k-1}|} [v_1, \cdots, v_n ], ~ \forall ~v_1, \cdots, v_n \in L, ~ \forall ~ k \in \nat.  \]

\begin{definition}
  A morphism $f: L \to L'$ between the \Linfs-algebras $(L, \{ \cdot, \cdots, \cdot \}_k)$ and $(L', \{ \cdot, \cdots, \cdot \}'_k)$ is a collection $f= ( f_n, ~ n \in \nat )$ of $n$-ary, multilinear, graded symmetric maps of degree $0$,
  \[ f_n: L^{\times n} \to L', ~ n \in \nat, \]
  such that
  \begin{equation} 
    \begin{split}
       & \sum_{i+j=n} \sum_{\sigma \in S_{i,j}} \alpha(\sigma, \mathbf{v}) f_{i+j+1}(\{ v_{\sigma(1)}, \ldots, v_{\sigma(i)} \}, v_{\sigma(i+1)}, \ldots, v_{\sigma(i+j)} ) \\
      = & \sum_{l=1}^k \sum_{\substack{n_1 + \cdots + n_l = n \\ 1 \leq n_1 \leq \cdots \leq n_l}} \sum_{\sigma \in \Sh(n_1, \ldots, n_l)} \alpha(\sigma, \mathbf{v}) \{ f_{n_1}(v_{\sigma(1)}, \ldots, v_{\sigma(n_1)}), \ldots, f_{k_l}(v_{\sigma(k-k_l+1)}, \ldots, v_{\sigma(k)}) \}',
    \end{split}
  \end{equation}
  for any $\mathbf{v} \in L^{\times k}$.
\end{definition}

\begin{definition}\label{defn:shLieAlgebroid}
  An \emph{\Linfs-algebroid} or a \emph{strong homotopy (SH) Lie-Rinehart algebra (\LRs-algebra}) is a pair $(\LD, \AD)$, where $\AD$ is a unital graded commutative $\field$-algebra and $\LD$ is an \Linfs-algebra with \textbf{$\field$-multilinear} $n$-brackets $\{ \cdot, \cdots, \cdot \}_n$, $n \in \nat$, which also possesses the structure of a graded $\A^\bullet$-module. Moreover, there is a family of $n$-ary $\field$-multilinear operations 
         \[\{ \cdot, \cdots, \cdot | \dash \}_n: L^{\times (n-1)} \times \A^\bullet \to \A^\bullet, ~ n \in \nat \]
      of degree $1$ such that
        \begin{enumerate}
           \item
             Each $\{ \cdot, \cdots, \cdot | \dash \}_n$ is $\A^\bullet$-multilinear subject to the Koszul sign rules and graded symmetric in the first $n-1$ entries and a derivation in the last entry. When $n=1$, the operation $\{| \dash \}_1 : \A^\bullet \to \A^\bullet$ does not depend on $L^\bullet$ and determines a derivation $d_A$ on $\A^\bullet$ of degree one. Hence $A=(\A^\bullet, d_A)$ is a unital commutative dga. When $n \geq 1$, the induced maps (of degree 0)
             \begin{equation}\label{eq:der2anchor}
               \alpha_n : L^{\times n} \to \Der^\bullet(A)[1], ~ (v_1, \ldots, v_n) \mapsto \{ v_1, \ldots, v_n | \dash \}_{n+1},
             \end{equation}
             form a morphism $\alpha$ of \Linfs-algebras between $\LD$ and shifted DGLA $\Der^\bullet(A)[1]$ of derivations of $A$. We call $\alpha$ as the \emph{$\infty$-anchor map} (or simply \emph{anchor map}) and $\alpha_n$ the $n$th anchor map of the \Linfs-algebroid $L$. Moreover, $\alpha_n$ is $A$-multilinear. 
           \item
             The brackets of $L$ and the anchor map satisfies the equality
             \begin{equation}
       \{v_1,\ldots,v_{n-1}, a \cdot v_n \}_n = \{v_1,\ldots,v_{n-1} | a \}_n \cdot v_n + (-1)^{|a|(|v_1|+\cdots + |v_{n-1}| +1 )} a \cdot \{ v_1,\ldots,v_n \}_n
     \end{equation}
       for any $n \in \nat$. Note that when $n=1$, this means $(L, d_L = \{ \cdot \}_1)$ is a dg-module over the dga $A=(A^\bullet, d_A)$.
        \end{enumerate}

\begin{remark}
  Unlike \cite{Luca}, we always write explicitly the underlying dga $A=(\A^\bullet,d_A)$ and we also call $(L^\bullet,d_L)$ as an \Linfs-algebroid over the dga $A$ if the operation $\{| - \}_1$ equals $d_A$. 
\end{remark}        
        
We now describe the (completed) Chevalley-Eilenberg dga of an \Linfs-algebroid. Instead of the multi-differential algebra structure on the uncompleted symmetric algebra used in \cite{Luca}, we use the completed symmetric algebra. The difference is minor.

Let  $\AD$ be a unital graded commutative $\field$-algebra and $\LD$ a graded $\AD$-module. Let $S^r_{\A}(L,\A)$ be the graded $\AD$ module of graded symmetric, $\AD$-multilinear maps with $r$ entries. A homogeneous element $\eta \in S^r_{\A}(L,\A)$ is a homogeneous, graded symmetric, $\field$-multilinear map 
  \[ \eta: (L^\bullet)^{\times r} \to \AD \] 
such that
  \[ \eta(a v_1, v_2, \ldots, v_r) = (-1)^{|a| |\eta|} a \eta(v_1, v_2, \ldots, v_r), ~ a \in \AD, v_1, \ldots, v_r \in \LD. \]
In particular, $S^0_{\A}(L, \A) = \AD$ and $S^1_{\A}(L, \A)= L^{\vee \bullet} :=\Hom_{\AD}(\LD,\AD)$. Define the completed symmetric algebra
  \[ \Shat_{\A}(L, \A) = \prod_{r \geq 0} S_{\A}^r (L, \A),  \]
where the product of two homogenous elements $\eta \in S^r_{\A} (L,\A)$ and $\eta' \in S^{r'}_{\AD} (\LD, \A)$ is given by the formula
  \[ (\eta \eta')(v_1, \ldots, v_{r+r'}) = \sum_{\sigma \in S_{r, r'}} (-1)^{|\eta'| ( |v_{\sigma(1)}| + \cdots + v_{\sigma(r)} )} \alpha(\sigma, \mathbf{v}) \eta (v_{\sigma(1)}, \ldots, v_{\sigma(r)}) \eta'(v_{\sigma(r+1)}, \ldots, v_{\sigma(r+r')}),  \]
for any $\mathbf{v}=(v_1, \ldots, v_{r+r'}) \in (\LD)^{\times (r+r')}$.  $\Shat_{\A} (L,\A)$ is a graded commutative unital algebra.

\begin{remark}
  Suppose $\LD = \AD \otimes_{\A^0} L^0$, where $\A^0$ and $L^0$ are the zeroth component of $\AD$ and $\LD$ respectively, and $L^0$ is a projective and finitely generated $A^0$-module. Then $\Shat_{\A} (L,\A) \simeq \AD \ctensor_{\A^0} \Shat_{\A^0}((L^0)^\vee)$ as $\AD$-modules, where $\ctensor$ is the complete tensor product with respect to the projective topology on $\Shat_{\A^0}((L^0)^\vee)$.
\end{remark}

The following theorem is an analogue of Theorem 12, \cite{Luca}.

\begin{theorem}\label{thm:Koszul}
  Let $\AD$ be an graded commutative unital $\field$-algebra and $\LD$ be a projective and finitely generated $\AD$-module. Then an \LRs-algebra structure on $(\LD, \AD)$ is equivalent to a degree one derivation on $\ShatL$.
\end{theorem}     

\begin{proof}[Sketch of proof]
We only recall the construction of the derivation from \cite{Luca}, which will be used later. For the detailed proof see the Appendix A of \cite{Luca}.

Since $\LD$ is projective and finitely generated, $\ShatL \simeq \Shat_{\A} (L^{\vee})$, the completed symmetric algebra generated by the $\AD$-module $L^{\vee \bullet}$, and any derivation is hence determined by its action on $\AD$ and $L^{\vee \bullet}$. Define degree one derivation $D_n$ on $\ShatL$ as follows. For any $a \in \AD$, set
  \[  (D_n a) (v_1, \ldots, v_n) := (-1)^{|a| (|v_1| + \cdots + |v_n|)} \{ v_1, \ldots, v_n | a \}_{n+1}, ~ \forall ~ v_1, \ldots, v_n \in \LD, ~ n \geq 0. \]
We have $D_n a \in S^n_{\A}(L^{\vee})$. Note that $D_0 a = d_A a$. For any $\eta \in L^{\vee \bullet}$, set
  \[ (D_n \eta) (v_1, \ldots, v_{n+1}) := \sum_{i=1}^{n+1} (-1)^\theta \{ v_1, \ldots, \widehat{v_i}, \ldots, v_{n+1}|\eta(v_i) \}_{n+1} + (-1)^{|\eta|} \eta(\{ v_1, \ldots, v_{n+1} \}),  \]
where $\theta:= |\eta| ( |v_1| + \cdots + \widehat{|v_i|} + \cdots + |v_{n+1}|) + |v_i|(|v_{i+1}| + \cdots + |v_{n+1}|)$, $v_1, \ldots, v_{n+1} \in L^\bullet, ~ n \geq 0$, and a hat $~\widehat{\cdot}~$ stands for omission. We have $D_n \eta \in S^{n+1}(L^{\vee})$. Note that $D_0$ restricted on $L^{\vee \bullet}$ is the differential $d_{L^\vee}$ induced from $d_L$ on $\LD$. 

The unique extension of $D_n$ as a degree one derivation on $\ShatL$ satisfies the higher Chevalley-Eilenberg formula:
  \begin{equation}
    \begin{split}
       &(D_n \eta) (v_1, \ldots, v_{n+r})   \\
       := & \sum_{\sigma \in \Sh(n,r)} (-1)^{|\eta|(|v_{\sigma(1)}| + \cdots + |v_{\sigma(n)}|)} \alpha(\sigma, \mathbf{v}) \{ v_{\sigma(1)}, \ldots, v_{\sigma(n)} | \eta(v_{\sigma(n+1)}, \ldots, v_{\sigma(n+r)})  \}_{n+1} \\
          &- \sum_{\tau \in \Sh(n+1,r-1)} (-1)^{|\eta|} \alpha(\tau, \mathbf{v}) \eta(\{ v_{\tau(1)}, \ldots, v_{\tau(n+1)}\}_{n+1}, v_{\tau(n+2)}, \ldots, v_{\tau(n+r)} ), 
     \end{split}
  \end{equation}
for any $\eta \in S^r_{\A}(L,\A)$, $\mathbf{v}=(v_1, \ldots, v_{n+k}) \in (\LD)^{\times (n+r)}$. 

It is proved in \cite{Luca} that $\sum_{j+k=n} D_j D_k = 0$ for all $n \geq 0$. Hence we can define the degree $1$ derivation $D = \sum_{n \geq 0} D_n$ on $\ShatL$ and $D^2=0$.

Conversely, any degree $1$ derivation $D$ on $\ShatL$ can be written as $D = \sum_{n \geq 0} D_n$, where $D_n$ maps $S^{r}_{\A}(L,\A)$ to $S^{r+n}_{\A}(L,\A)$. For any $a \in \AD$ and $v_1, \ldots, v_n \in \LD$, set
  \begin{equation}\label{eq:CE2der} 
     \{  v_1, \ldots, v_{n-1} | a \}_n := (-1)^{|a|(|v_1|+ \cdots + |v_{n-1}|)} (D_{n-1} a) (v_1, \ldots, v_{n-1}) \in \AD 
  \end{equation}
and let $\{ v_1, \ldots, v_n \}_n$ be the unique element in $\LD$ satisfying
  \begin{equation}\label{eq:CE2bracket}
    \begin{split}
      \eta(\{ v_1, \ldots, v_n \}_n) :=    &   (-1)^{|\eta|} \sum_{i=1}^n \sum_{i=1}^n (-1)^{|v_i|(|v_1| + \cdots + |v_{i-1}|)} D_{n-1} (\eta(v_i)) (v_1, \ldots, \widehat{v_i}, \ldots, v_n) \\
       & - (-1)^{|\eta|} (D_{n-1} \eta)(v_1, \ldots, v_n), 
    \end{split}
  \end{equation}
for any $\eta \in L^{\vee \bullet}$ (here we use the projectivity and finiteness of $\LD$).
\end{proof}

\begin{definition}
The dga $(\ShatL, D)$ determined by an \Linfs-algebroid $(\LD, \AD)$ in Theorem \ref{thm:Koszul} is called the \emph{(completed) Chevalley-Eilenberg dga} of $\LD$. 
\end{definition}

\begin{remark}
If the dga structure $A=(\AD,d_A)$ is fixed in advance, an \Linfs-algebroid structure over $A$ on $\LD$ is equivalent to a derivation $D$ on $\ShatL$, whose zeroth component $D_0$ acts on $\AD$ as $d_A$ and acts on $L^{\vee \bullet}$ as $d_{L^\vee}$.
\end{remark}

\begin{comment}        
        \[\{ \cdot, \cdots, \cdot | \cdot \}: L^{\times (k-1)} \times \A^\bullet \to \A^\bullet \]
      There is an \emph{anchor map} $\alpha$ which is a $L_\infty$-morphism from $L^\bullet$ to the DGLA $\Der^\bullet(A)$ of derivations of $A$ such that
        
    \begin{itemize}  
      \item
     Each component 
     \begin{displaymath}
       \alpha_n : (L^\bullet)^{\otimes n} \to \Der^{\bullet}(A), \quad \forall ~ n \in \nat,
     \end{displaymath}
     of degree $1-n$ (called $n$th anchor) is $A$-linear. That is, if we define $\alpha( \cdot, \cdots, \cdot | - ) : (L^\bullet)^{\otimes n} \otimes_\field \A^\bullet \to \A^\bullet$ by
       \[\alpha_n( v_1, \cdots, v_n | a ) = \alpha_n(v_1, \cdots, v_n)(a), \quad v_1, \cdots, v_n \in L^\bullet, ~ a \in \A^\bullet, \]
       then it is $\A^\bullet$-linear in the first $n$-entries and a derivation in the last entry.
   \item
     Each $n$-bracket $l_n$ is an $A$-derivation on each of its arguments via the anchor map $\alpha$, that is,
     \begin{equation}
       [v_1,\ldots,v_{n-1}, a \cdot v_n]_n = \alpha_{n-1}(v_1,\ldots,v_{n-1}|a) \cdot v_n + (-1)^{|a|(|v_1|+\cdots + |v_{k-1}| -n)} a \cdot [v_1,\ldots,v_n]_n
     \end{equation}
     for any homogeneous 
   \end{itemize}
\end{comment}
 
\end{definition}

\subsection{\Linfs-algebroid of the formal neighborhood}  

We now repackage the differential $\Dnormal$ in the formula \eqref{eq:Dnormal_Final} in the language of \Linfs-algebroids. The underlying dga $A$ is the Dolbeault dga ($\AOD(X), \dbar$), hence the shifted DGLA $\A^{0, \bullet+1}_X(TX)$ of shifted Dolbeaut complex of the tangent bundle $TX$ equipped with the usual Lie bracket is a shifted dg-Lie subaglebra inside $\Der^\bullet(A)[1]$. The underlying $A$-module of the \Linfs-algebroid is the Dolbeault complex $(L^\bullet,d_L) = (\AOD_X (N), \dbar)$ of the normal bundle $N$ of $X$ inside $Y$. The structure maps are given as follows. First of all, we compute the values of the anchors on $S^r_{A}(L,A) = \A^{0,0}_X(S^r N)$, using \eqref{eq:Dnormal_zero}, \eqref{eq:der2anchor} and \eqref{eq:CE2der}, and get the recursive formulas,
\begin{gather}
  \alpha_1 = R^\top_1 = \beta:  \A^{0,0}_X(N) \to \A^{0,1}_X(TX) \subset \Der^1(A),   \\
  \alpha_n = R^\top_n + \sum_{\sigma \in \Sh(n-1,1)} \shape \circ (\alpha_{n-1} \times 1) \circ \sigma:   \A^{0,0}_X(N)^{\times n} \to \A^{0,1}_X(TX) \subset \Der^1(A), \quad n \geq 2,
\end{gather}
where $\sigma$ acts as permutation on the $n$ $\A^{0,0}_X(N)$-factors and $\shape : TX \otimes \normal \to TX$ is the shape operator.) Then we extend $\alpha_n$ to an $A$-multilinear map from $\AOD_{X}(N) ^{\times n}$ to  $\A^{0,\bullet+1}_X(TX)$ subject to the Koszul sign rules.

Similarly, by \eqref{eq:CE2bracket}, \eqref{eq:Dnormal_Final} and the formulas above for $\alpha$, we get the formulas for the $n$-ary brackets,
\begin{gather}
  \ell_n = R^\bot_n + \sum_{\sigma \in \Sh(n-1,1)} \normalconn \circ (\alpha_{n-1} \times 1) \circ \sigma: \A^{0,0}_X(N)^{\times n} \to \A^{0,1}_X(N), \quad n \geq 2,
\end{gather}
where the connection $\normalconn$ on the normal bundle is considered as a map
\begin{displaymath}
  \normalconn: \AOD_X(TX) \otimes_{\complex} \AOD_X(\normal) \to \AOD_X(\normal)
\end{displaymath}
of degree zero. Then we can extend the brackets uniquely such that it satisfies the condition $(1)$ of Definition \ref{defn:shLieAlgebroid}.

\begin{comment}
\begin{remark}
  One can also interpret the homomorphism
  \begin{displaymath}
    \widetilde{\pi}^* : (\AOD(\Shat^\bullet(\conormal)),\Dnormal) \to (\AOD(\Shat^\bullet(T^*Y),D)
  \end{displaymath}
  defined in \eqref{eq:tildepi} as an $L_\infty$-morphism from $\AOD(TY[-1])$, thought of as an $L_\infty$-algebroid with zero anchor map, to $\AOD(\normal[-1])$. But note that this morphism commutes with anchor maps only up to homotopy.
\end{remark}
\end{comment}

%%%%%%%%%%%%%%%%%%%%%%%%%%%%%%%%%%%%%%%%%%%%%%

\bibliographystyle{amsalpha}
\addcontentsline{toc}{chapter}{Bibliography}
\bibliography{Dolbeault_dga_Bib}

\end{document}